\documentclass{article}
\usepackage[utf8]{inputenc}
\usepackage[T1]{fontenc}
\usepackage{lmodern}
\usepackage{graphicx}
\usepackage{a4wide}
\usepackage{amsmath, amssymb}
\usepackage{amsthm,xcolor,hyperref}
\usepackage{paralist}
\usepackage{cleveref,comment}
\usepackage{multidef}
\usepackage[margin=15mm,bottom=23mm]{geometry}
\usepackage{xspace}
\usepackage{comment}
\usepackage[T2A,T1]{fontenc}
\usepackage[russian,english]{babel}

\usepackage[linesnumbered,ruled,vlined]{algorithm2e}
\SetAlgoCaptionSeparator{}

\multidef[prefix=cal]{\ensuremath{\mathcal{#1}}\xspace}{A-Z}
\multidef[prefix=bf]{\ensuremath{\mathbf{#1}}\xspace}{A-Z}
\multidef[prefix=bf]{\ensuremath{\mathbf{#1}}\xspace}{a-z}
\multidef[prefix=sf]{\ensuremath{\mathsf{#1}}\xspace}{a-z}
\multidef[prefix=sf]{\ensuremath{\mathsf{#1}}\xspace}{A-Z}
\multidef[prefix=bb]{\ensuremath{\mathbb{#1}}\xspace}{A-Z}
\multidef[prefix=fr]{\ensuremath{\mathfrak{#1}}\xspace}{A-Z}
\multidef[prefix=scr]{\ensuremath{\mathscr{#1}}\xspace}{A-Z}

\theoremstyle{plain}
\newtheorem*{mainthm}{Main Theorem}
\newtheorem*{unnumbered corollary}{Corollary}
\newtheorem{theorem}{Theorem}

\newtheorem{lemma}[theorem]{Lemma}
\newtheorem{proposition}[theorem]{Proposition}
\newtheorem{corollary}[theorem]{Corollary}

\theoremstyle{definition}

\newtheorem{remark}[theorem]{Remark}

\newcommand{\DC}{\mathsf{DC}}
\newcommand{\NGr}{\mathsf{N}}
\newcommand{\EC}{\mathsf{EC}}
\newcommand{\WP}{\mathsf{WP}}
\newcommand{\GL}{\mathsf{GL}}
\newcommand{\Bounds}{\mathsf{Bounds}}
\newcommand{\ModC}{\mathsf{ModC}}
\newcommand{\QWP}{\mathsf{QWP}}
\newcommand{\Lab}{\mathsf{GetLabel}}
\newcommand{\Var}{\mathsf{Var}}

\newcommand{\mat}{\mathsf{mat}}

\newcommand{\Id}{\bfI\bfd}
\renewcommand{\gcd}{\mathsf{gcd}}
\newcommand{\Time}{\mathsf{Time}}
\newcommand{\bbCay}{\textsf{Cay}}
\newcommand{\thh}{\textsuperscript{th}\xspace}
\newcommand{\ie}{\emph{i.e.}}

\newcommand{\inv}{^{-1}}
\newcommand{\sflabel}{\mathsf{label}}
\let\phi\varphi

\title{The average-case complexity of the Word Problem for groups of matrices over~$\bbQ$ is linear}

\author{
Frédérique Bassino, \small{\url{bassino@lipn.fr}}\\
\small{Univ. Sorbonne Paris Nord, CNRS, LIPN, UMR 7030, F-93430 Villetaneuse, France}% 
\and
Vincent Jugé, \small{\url{vincent.juge@univ-eiffel.fr}}\\
\small{Univ. Gustave Eiffel, CNRS, LIGM, F-77454 Marne-la-Vallée, France}% 
\and
Cyril Nicaud, \small{\url{cyril.nicaud@univ-eiffel.fr}}\\
\small{Univ. Gustave Eiffel, CNRS, LIGM, F-77454 Marne-la-Vallée, France}%
\and
Pascal Weil, \small{\url{pascal.weil@cnrs.fr}}\\
\small{CNRS, Univ. Sorbonne Paris Nord, LIPN, UMR 7030, F-93430 Villetaneuse, France}\\
\small{CNRS, ReLaX, IRL 2000, Siruseri, India}%
}
\date{\today}

\begin{document}

\maketitle

%%%%%%%%%%%%%%%%%%%
\section{Introduction}

The Word Problem is one of the fundamental decision problems in group theory. 
Introduced by Dehn in the early 20th century \cite{1911:Dehn}, it asks whether a word written on the set of generators of a given group $G$ evaluates to the identity.
In this paper, we establish that the Word Problem in a finitely generated group of invertible matrices over $\bbQ$ can be solved in linear average time.

The Word Problem is not decidable in general:
there exist finitely generated, and even finitely presented groups with undecidable Word Problem (see Novikov \cite{1955:Novikov} and Boone \cite{1959:Boone}).
However, it is known to be decidable in many important classes of groups, for instance in automatic groups (see Epstein \emph{et al.} \cite{1992:Epsteinetal}), which include finite, free, hyperbolic or braid groups, in finitely presented residually finite groups (Simmons \cite{1973:Simmons}), in 1-relator groups (see Magnus \emph{et al.} \cite[Theorem 4.14]{1966:MagnusKS} and also Lyndon and Schupp \cite{1977:LyndonSchupp}), among others.
In a group of matrices with rational coefficients, the Word Problem is trivially decidable:
one simply computes the matrix product and compares it to the identity. Lipton and Zalcstein showed that, in that case, it can be solved in deterministic logspace \cite{1977:LiptonZalcstein}, and hence in polynomial time.
Recently, Olshanskii and Shpilrain observed that the worst-case complexity of the problem is $\calO(n\log^2n)$ in groups of matrices with integer coefficients (where $n$ is the length of the input word).
They also showed that it can be solved in almost linear time in groups of upper triangular matrices with integer coefficients, and hence in finitely generated nilpotent groups \cite{2025:OlshShp}.

Our results, presented below, include $\calO(n\log^2 n)$ worst-case complexity statements for groups of matrices over rings of algebraic integers or over $\bbQ$ (confirming in the latter case an assertion in \cite{2025:OlshShp}, given there without proof).
Our main results however deal with the average-case complexity of the Word Problem.
We note that the discussion of average-case complexity in the context of algorithms on groups is fairly recent, see \cite{2003:KapovichMyasnikovShpilrainSchupp,2005:KapovichMyasnikovSchuppShpilrain}, which deal with the Word Problem, among other decision problems, and \cite{2023:Shpilrain,2024:RoyVenturaWeil} which deal with other problems in free groups.

To fix notation, we let~$A$ be a finite non-empty set, $d$ be an integer at least 2, and $\mat$ be an arbitrary function from $A$ to $\GL_d(\bbQ)$, where~$\GL_d(\bbQ)$ denotes the group of invertible~$d \times d$ matrices with rational coefficients.

Let $\bar A = \{\bar a : a\in A\}$ be a set of formal inverses of the elements of $A$, and let $\tilde A = A \sqcup \bar A$.
Let $\tilde A^*$ denote the \emph{free monoid} over $\tilde A$, that is, the set of finite sequences of elements of~$\tilde A$.
Also, let $F(A)$ denote the \emph{free group} over $A$, which we identify with the set of reduced words in $\tilde A^*$, that is, the words without factors of the form $a \bar a$ or~$\bar a a$, with $a\in A$.

For each $a\in A$, let $\mat(\bar a) = \mat(a)\inv$.
This defines a natural extension of the map $\mat$ to a monoid morphism from $\tilde A^*$ and, indeed, to a group morphism from $F(A)$ to~$\GL_d(\bbQ)$.
Let $G$ denote the range of $\mat$, \ie, $G = \mat(\tilde A^*)$.
Thus, the set $\mat(\tilde A)$ generates $G$ as a monoid, and the set $\mat(A)$ generates it as a group.

We consider two closely related algorithmic problems:
the \emph{Exact Computation Problem}~$\EC_A$ and the \emph{Word Problem}~$\WP_A$ in~$G$ (relative to $A$).
When given as input a word~$w \in \tilde A^*$,~$\EC_A$ consists in computing the matrix~$\mat(w)$, and~$\WP_A$ consists in deciding whether~$\mat(w)$ is the identity matrix~$\Id$.
Depending on the context, the input word~$w$ for~$\EC_A$ or for $\WP_A$ can be a word in $\tilde A^*$, or a reduced word, that is, an element of $F(A)$.

Our main result is the following (see \Cref{thm:WP linear average,prop: WP linear over Q}). 

\begin{mainthm}
Let $A$ be a fixed finite set of elements of $\GL_d(\bbQ)$.
The word problem~$\WP_A$ has linear time average-case complexity, both in the case where inputs are chosen uniformly at random among words of length $n$ over the alphabet~$\tilde A$ and in the case when they are chosen uniformly at random among reduced words of length $n$.

Under very mild hypotheses, this bound is tight.
\end{mainthm}

Kapovich et al.~\cite{2005:KapovichMyasnikovSchuppShpilrain} established average-case complexity results for the Word Problem in finitely presented groups, under the hypothesis it is solvable.
In particular, \cite[Example~1.8]{2005:KapovichMyasnikovSchuppShpilrain} applies to finitely generated linear groups over a field of characteristic~0, and combined with Lipton and Zalcstein's logspace algorithm \cite{1977:LiptonZalcstein}, proves linear average-case complexity of the Word Problem, but only under the extra hypothesis that the group has a finite-index subgroup with a non-elementary word-hyperbolic quotient. Yet no virtually solvable group satisfies this, since a non-elementary hyperbolic group always contains a free subgroup of rank two~\cite[Chapter~8, Theorem~37]{1990:GhysHarpe}, which virtually solvable groups never do. In contrast our results provide an algorithm with linear average-case complexity for every finitely generated group of matrices over $\bbQ$.

It is well known that the coefficients of a product of $n$ matrices in $\tilde A$, say, with integer coefficients, can grow exponentially in $n$.
For this reason, we use the bit-complexity model, where integers are represented by their binary expansions and arithmetic operations require more than constant time.
Rationals are identified with reduced fractions, and are represented by the pair of binary expansions of their numerator and denominator, see \Cref{sec: complexity,sec: rational}.

A naive algorithm (compute one by one the products corresponding to longer and longer prefixes of the input word) solves $\EC_A$ and $\WP_A$ in (worst-case) time $\calO(n^2)$, where $n$ is the length of the input word.
A standard divide-and-conquer strategy lowers this complexity to $\calO(n\log^2 n)$.
The latter result was observed by Olshanskii and Shpilrain for subgroups of $\GL_d(\bbZ)$ \cite[Proposition 2]{2025:OlshShp}.
These authors also asserted the result for subgroups of $\GL_d(\bbQ)$. This algorithm and its complexity are discussed in \Cref{sec: first algos} for matrices with integer coefficients, and in \Cref{sec: rational} for matrices with rational coefficients.

We first prove our Main Theorem for matrices in $\GL_d(\bbZ)$ (\Cref{sec: a better algorithm}), which is technically easier to handle.
The proof there uses the fairly standard idea of selecting an appropriate modulus $m$, which turns out to depend on the length of the input word and the growth of the group generated by $\mat(A)$:
one computes, in linear time, the value of $\mat(w)$ mod $m$ and checks whether that is the identity matrix. If it is not, then $\mat(w) \ne \Id$;
if it is, one runs the divide-and-conquer algorithm mentioned above.

The fact that our algorithm effectively solves $\WP_A$ is trivial.
The proof that it does so in linear average time is quite involved, and occupies much of the paper.
It relies on results on random walks on infinite groups, discussed in \Cref{sec: random walks on Cayley graphs}. 
The key ingredients are large-deviation and Gaussian-type bounds on the transition probabilities of the walk (\Cref{thm:hebisch}), which establish transience and, via a coupling at last-passage times, yield an auxiliary non-backtracking walk whose successive last-passage times have i.i.d. increments (\Cref{lemma:tau n}).
A Berry-Esseen central limit theorem (\Cref{thm:esseen-berry}) then controls these increments.
The case of non-backtracking random walks, needed to handle the case where input words for $\WP_A$ are taken to be reduced, requires extra care and fine-tuning.

We then explain (\Cref{sec: larger rings}) how the algorithm must be adapted to deal with matrices in $\GL_d(\bbQ)$ (\Cref{sec: rational}) and, also, to handle groups of matrices over algebraic number fields any characteristic 0 ring $R$ such that the additive group $(R,+)$ has finite dimension over $\bbC$ (\Cref{sec: ring of integers}).

Several special cases and corollaries are worth mentioning (see \Cref{prop: upper triangular} and \Cref{cor: polycyclic}).

\begin{unnumbered corollary}
The Word Problem $\WP_A$ has linear worst-case complexity for nilpotent groups.

For polycyclic groups, $\WP_A$ has worst-case complexity $\calO(n\log^2 n)$ and linear average-case complexity.
\end{unnumbered corollary}

The paper closes with a discussion (\Cref{sec: lower bound}) of the fact that no algorithm can solve the Word Problem in $\GL_d(\bbZ)$ in sublinear time, whether this is a worst-case or an average-case complexity.

%%%%%%%%%%%%%%%%%%%%
\subsubsection*{Notational Convention}
Throughout the paper, $\log$ stands for the logarithm in base 2 and $\log^hx$ stands for the $h$\thh power of $\log x$, \ie, $(\log x)^h$.

Let $m$ be a positive integer.
When $n\in \bbZ$, let $n_{\equiv m}$ denote the residue of~$n$ modulo~$m$, \ie, the least integer~$r \geqslant 0$ such that~$r$ is congruent to $n$ modulo $m$;
this is the projection of $n$ in $\bbZ/m\bbZ$.
When~$M$ is a matrix with integer coefficients and~$m$ is a positive integer, let~$M_{\equiv m}$ denote the projection of~$A$ in~$\GL_d(\bbZ/m\bbZ)$. More generally, if $X$ is a set of such matrices, then $X_{\equiv m}$ denotes the set $\{M_{\equiv m} \colon M\in X\}$.

When $f(n)$ and $g(n)$ are functions defined on $\bbN$, we write
\begin{itemize}
 \item {}$g(n) = \calO(f(n))$ if there exists a constant $C>0$ such that $g(n) \leqslant C\,f(n)$ for all $n$ large enough;
 \item {}$g(n) = \Omega(f(n))$ if $f(n) = \calO(g(n))$;
 \item {}$g(n) = \Theta(f(n))$ if we have both $g(n) = \calO(f(n))$ and $g(n) = \Omega(f(n))$.
\end{itemize}

%%%%%%%%%%%%%%%%
\section{Preliminaries}

The proof that our algorithm solves $\WP_A$ in linear average time for matrices over $\bbZ$ (\Cref{sec: a better algorithm}) is based on deep results in group theory and in probability theory.
These results, along with some basic notions on complexity and on the representation of integers, are discussed in this section.

%%%%%%%%%%%%%%%%
\subsection{Complexity model and basic facts}\label{sec: complexity}

To evaluate the complexity of an algorithm~$\mathcal{S}$, solving~$\WP_A$ or any other problem, we proceed as follows. Let $\Time(\mathcal{S},w)$ denote the running time of $\mathcal{S}$ on input word $w$, measured as the number of elementary operations performed by $\mathcal{S}$.
Throughout the paper, we work in the \emph{bit-complexity model}, where an elementary operation is an arithmetic or bitwise comparison operation on individual bits. As a result, the cost of each operation depends on the bit lengths of its operands.
The \emph{worst-case complexity} of~$\calS$ is the function on integers given by~$\sfC_\textsf{wc}^{(\calS)}(n) = \max\{\Time(\calS,w) \colon |w| = n\}$ while
its average complexity~$\sfC_\textsf{ac}^{(\calS)}(n)$ is the expected value of~$\Time(\mathcal{S},w)$, when~$w$ is chosen uniformly at random over all inputs of length~$n$.
As usual, both functions are considered up to asymptotic equivalence, when~$n$ tends to infinity.

\paragraph{Bit length of integers}
Our algorithms manipulate large integers, and therefore the cost of reading and performing arithmetic operations on such integers must be taken into account.
Typically, we deal with morphisms $\mat\colon A^* \to \GL_d(\bbZ)$ for which the entries of a matrix $\mat(w)$ may grow exponentially fast in the length $|w|$ of $w$.
We therefore work in the standard bit-complexity model.
As is customary in such situations, we assume that accessing the $i$\thh letter of a word, or the $i$\thh entry in an array, requires constant time, and in effect, that operations on indices (such as positions in the input word or in auxiliary data structures) are done in constant time.
By contrast, matrix entries are represented in binary, and the cost of arithmetic operations on them depends on their bit length.

More precisely, we define the \emph{length} $\ell(n)$ of an integer $n$ to be the number of bits in its signed binary representation, that is,
$$ \ell(0)=1, \qquad \ell(n) = 2+\lfloor \log |n| \rfloor \quad\text{for } n\neq 0$$
(where the first bit encodes the sign, positive or negative). 
As a result, the absolute value of integers of length at most~$k$ is less than~$2^{k-1}$.

We also define the \emph{length} $\ell(M)$ of a $d\times d$ matrix $M$ to be the maximum of the lengths of its entries.
Since $d$ is fixed, the total number of bits required to store $M$ is $\Theta(\ell(M))$;
consequently, measuring the length of a matrix by $\ell(M)$ is equivalent, up to constant factors, to measuring its actual storage length.

We freely use the following facts.
Item~\eqref{item: HvdH} is due to Harvey and van der Hoeven~\cite{harvey2021integer};
and Items~\eqref{item: residue mod m} and~\eqref{item: arithmetic mod m} follow from Item~\eqref{item: HvdH} and \cite[Theorem 9.8 and Corollary 9.9]{vzGathenGerhard}.
The other items are elementary observations.

\begin{proposition}\label{prop: basic facts}
Let~$m$ and~$n$ be integers and let $L = \max(\ell(m),\ell(n))$.
\begin{enumerate}[(i)]
\item\label{item: comparing integers} Comparing $m$ and $n$ can be done in time $\calO (\min(\ell(m),\ell(n)))$.

\item The sum $m+n$ has length at most $\ell(m+n) \leqslant 1 + L$ and is computed in time~$\calO(L)$.

\item\label{item: HvdH} The product~$mn$ has length~$\ell(mn) \leqslant \ell(m) + \ell(n)$ and is computed in time $\calO(L \log L)$.

\item\label{item: residue mod m} The residue of~$n$ modulo~$m$, \ie, the least integer~$r \geqslant 0$ such that~$r \equiv n \pmod{m}$, can be computed in time~$\calO(\ell(n) \log \ell(n))$.

\item\label{item: arithmetic mod m} Every arithmetic operation in $\bbZ/m\bbZ$ can be performed in time $\calO(\log m\, \log\log m)$.

\item\label{fact:matrix-product} Let~$M$ and~$N$ be~$d \times d$ matrices with integer coefficients, and let $L' = \max(\ell(M), \ell(N))$.
Then $\ell(MN) = \calO(L')$, and its coefficients can be computed in time $\calO(L'\log L')$.

\item\label{item: length of mat(w)} If $w$ is a non-empty finite word in $\tilde A^*$ and $d \geqslant 2$, then
$$\ell(\mat(w)) \leqslant 1 + |w| \cdot \max_{a\in \tilde A}\{\ell(\mat(a))\} + (|w|-1) \log d \leqslant |w| \,(\max_{a\in \tilde A}\{\ell(\mat(a))\} + \log d)$$

\item If $k$ is a constant, then multiplying an integer $n$ by $k$ can be done in time $\calO(\ell(n))$.
Similarly, if $K$ is a fixed $d\times d$ integer matrix, $d$ being a fixed integer, then the product $MK$ can be computed in time $\calO(\ell(M))$. 
\end{enumerate}
\end{proposition}

\noindent {\bf Throughout the paper}, the integer $d \geq 2$, the set $A$ and the map $\mat\colon A\to \GL_d(\bbQ)$ are fixed.
We assume that the extension of $\mat$ to $\tilde A$ has also been precomputed, and we let $\sfB$ denote the maximum length of the matrices in $\mat(\tilde A)$.

%%%%%%%%%%%%%%%%%%%%%%%%%%%%%%%%%%%%%%%%%%%%%%
\subsection{Growth of groups}

We will use a famous result on the dichotomy between two types of growth for finitely generated linear groups.
Let $H$ be a finitely generated linear group and let $S$ be a finite set of generators.
Let $B_S(g,n)$ denot ethe set of elements of $H$ which can be written as $gh$, where $h$ is a product of length at most $n$ of elements of $S \cup S\inv$.
For $h \in H$, write $\|h\|$ for the least such $n$, that is, the word-length of $h$ relative to $S$.
We say that $H$ has \emph{exponential growth} if there exists a constant $c > 1$ such that $|B_S(1,n)| = \Omega(c^n)$.
If, instead, there exists an integer $\delta \geqslant 0$ such that $|B_S(1,n)| = \Theta(n^\delta)$, we say that $H$ has \emph{polynomial growth} of order~$\delta$.%
\footnote{The general definition of polynomial growth for an $S$-generated group $H$ is to have $|B_S(1,n)| \leqslant P(n)$ for some polynomial $P$.
In the case of linear groups, Gromov's theorem \cite{1981:Gromov} and the Bass–Guivarc'h formula \cite{1972:Bass,1973:Guivarch} imply that $|B_S(1,n)| = \Theta(n^\delta)$ for some integer $\delta \geqslant 0$.} 
It is easily verified that having exponential growth, or polynomial growth of order $\delta$, does not depend on the choice of the set $S$.

\begin{theorem}\label{prop:tits}
Let~$H$ be a finitely generated linear group over a field of characteristic 0. Then $H$ has either exponential growth or polynomial growth.
\end{theorem}

\begin{proof}
This result follows from three deep results.
In 1968, Milnor showed that a solvable group which is not polycyclic has exponential growth~\cite{1968:Milnor};
and Wolf showed that a polycyclic group has either polynomial growth or exponential growth, and that it has polynomial growth if and only if it is virtually nilpotent~\cite[Theorem 4.3]{1968:Wolf}.
In particular, a virtually solvable group either has exponential growth, or it has polynomial growth and it is polycyclic and virtually nilpotent.

Finally, Tits' alternative (1972, \cite[Theorems~1 and~2]{1972:Tits}) shows that a linear group is either virtually solvable, or contains a copy of $F_2$ (and hence has exponential growth), and this completes the proof.
\end{proof}

%%%%%%%%%%%%%%%%%%%%%%%%%%%%%%%%%%%%%%%%%%%%%%
\subsection{Random walks on Cayley graphs}\label{sec: random walks on Cayley graphs}

Let $H$ be a finitely generated group and let $S$ be a finite generating set of $H$.
The \emph{Cayley graph} of $H$ relative to $S$, written $\bbCay(H,S)$, is the directed graph whose vertices are the elements of $H$, with an edge from $h$ to $h'$ if there exists an element $s \in S\cup S\inv$ such that $h' = h\,s$.
With this notation, $\|h\|$ is the distance from the identity to $h$ in $\bbCay(H,S)$.

In this paper, a \emph{random walk} on $\bbCay(H,S)$ is a stochastic process $(X_n)_n$ with values in $H$, started at the identity, such that $X_{n+1}=X_nY_{n+1}$, where $(Y_n)_n$ is a sequence of independent identically distributed (i.i.d.) random variables with distribution $\mu$ supported on $S\cup S\inv$.
The \emph{support} of the random walk is the support of $\mu$, \ie\ $\operatorname{supp}(\mu)=\{g\in S\cup S\inv:\mu(g)>0\}$.
Since $S$ is finite, all random walks considered here have finite support.
A random walk is said to be \emph{symmetric} if $\mu(s) = \mu(s\inv)$ for every $s\in S$, and \emph{simple} if $\mu$ is the uniform distribution on $S\cup S\inv$.

If $H = F(A)$ is the free group on a set $A$ and $S = A$, an edge from $h$ to $h'$ uniquely determines the letter $a \in \tilde A$ such that $h' = h\,a$, so paths in $\bbCay(F(A),A)$ can naturally be labeled by words in $\tilde A^*$.
In particular, after $n$ steps, a simple random walk on $\bbCay(F(A), A)$ yields a word chosen uniformly at random in~$\tilde A^n$.

We also define a \emph{simple non-backtracking random walk} on $\mathrm{Cay}(F(A),A)$ to be a sequence of random variables $(X_n)_{n\geqslant 0}$ with values in $F(A)$ satisfying the following.
If the reduced word representing $X_n$ ends with the letter $a\in\widetilde A$, then $X_{n+1}=X_n b$, where $b$ is chosen uniformly at random from $\widetilde A\setminus\{\bar a\}$.
After $n$ steps, $X_n$ is thus uniformly distributed among all reduced words of length $n$ over $\widetilde A$.

We discuss below several results from the literature, and their consequences in the context we are interested in.
A first result yields upper bounds on the probability of reaching a given group element after~$n$ steps in a symmetric random walk in the Cayley graph of a finitely generated group with polynomial growth, and stronger bounds for groups with exponential growth.
The first item follows directly from \cite[Remark 1]{1991:HebischSaloffCoste} (following Theorem 5.3 in that paper) and the second item is Inequality (13) of \cite[Theorem 5.1]{1991:HebischSaloffCoste}.%
\footnote{Observe that \cite{1991:HebischSaloffCoste} requires the random walks under consideration to have finite support.
As mentioned above, this is always the case for the random walks on $\bbCay(H,S)$.}

\begin{theorem}\label{thm:hebisch}
Let $H$ be a finitely generated group and let $S$ be a finite generating set of $H$.
Let $(X_n)_n$ be a symmetric random walk on the Cayley graph $\bbCay(H,S)$, whose support generates $H$.
\begin{enumerate}[(i)]
\item\label{item: expo growth} If~$H$ has exponential growth, there exists a constant~$c > 0$ such that
\[\bbP[X_n = g] \leqslant c \exp\left(-\frac{n^{1/3}}{c}\right)\]
for all integers~$n \geqslant 0$ and all~$g \in H$.
\item\label{item: poly growth} If~$H$ has polynomial growth of order~$\delta \geqslant 1$, there exists a constant~$c > 0$ such that
\[\bbP[X_n = g] \leqslant c n^{-\delta/2} \exp\left(- \frac{\|g\|^2}{c n}\right)\]
for all integers~$n \geqslant 1$ and all~$g \in H$.
\end{enumerate}
\end{theorem}

A second result of interest, stated in \Cref{lemma:tau n} below, links simple random walks to non-backtracking random walks.
We know of two instances of this type of construction, namely \cite[p.~2296]{2019:Hermon} and \cite[Fact~4.1(1a)]{2018:HermonPeres}), both used without proof.
We are not aware of a published proof of either statement, and therefore provide a complete proof of the statements we require (notably through the use of the Berry--Esseen estimate of \Cref{thm:esseen-berry}), including a quantitative moment control condition.
Recall that the \emph{$h$\thh moment} of a real-valued random variable $X$ is the expectation of $X^h$.

\begin{lemma}\label{lemma:tau n}
Let $A$ be a finite set with at least two elements and let $(Y_n)_n$ be a simple random walk on the Cayley graph $\bbCay(F(A),A)$ of the free group over $A$.

For each $n\geqslant 0$, let $\tau_n$ be the time of last passage of the random walk on the sphere of radius $n$, that is, $\tau_n = \sup\{k : \|Y_k\| = n\} \in \bbN \cup \{\pm\infty\}$ (with the convention $\sup\varnothing = -\infty$).
Almost surely, each $\tau_n$ is finite.

Moreover, almost surely, $(Y_{\tau_n})_{n \geqslant 0}$ is a simple non-backtracking random walk on $\bbCay(F(A),A)$, and the differences~$\tau_{n+1} - \tau_n$ ($n \geqslant 1$) are positive, independent and identically distributed, with positive variance and finite third moment.
\end{lemma}

\begin{proof}
Let $\alpha = |\tilde{A}|$, and let $Z_k = \|Y_k\|$:
$\tau_n$, which is the time of last passage of $(Y_k)_k$ in the ball of radius $n$, coincides with the time of last passage of $(Z_k)_k$ at $n$.
In other words, $\tau_n = \sup\{k \colon \|Y_k\| \leqslant n\} = \sup\{k \colon Z_k = n\}$.
For all $\ell \geqslant 1$ and $k \geqslant 0$, conditionally to event $Z_k = \ell$, we have $Z_{k+1} = \ell+1$ with probability $p = (\alpha-1) / \alpha$, and $Z_{k+1} = \ell-1$ with probability $q = 1/\alpha$.
If $\ell = 0$, we have $Z_{k+1} = \ell+1$ with probability $1$. 
Hence $(Z_k)_{k \geqslant 0}$ is a Markov chain on the set $\bbN$ of non-negative integers, starting at $0$, which moves like a simple random walk on $\bbZ$ with parameters $p$ and $q$, except that it is reflected at $0$. 

Now let $(\varepsilon_k)_{k \geqslant 1}$ be i.i.d. random variables such that $\bbP[\varepsilon_k = 1] = p$ and $\bbP[\varepsilon_k = -1] = q$ for all $k \geqslant 1$.
The sums $X_k = \varepsilon_1 + \varepsilon_2 + \cdots + \varepsilon_k$ form a simple random walk on $\bbZ$ with $X_0=0$ and parameters $p$ and $q$.

Let $\bar\tau_n$ be the time of last passage of $(X_k)_k$ at $n$.
We may realize $Z$ by setting $Z_{k+1}=Z_k+\varepsilon_{k+1}$ if $Z_k\geqslant1$ and $Z_{k+1}=1$ if $Z_k=0$. An immediate induction shows that $Z_k\geqslant X_k$ for all $k$.
Since $\alpha\geqslant4$,  we have $p>q$, and hence $X_k$ tends to $\infty$ almost surely by the law of large numbers, and so does $Z_k$.
Since both $X$ and $Z$ start at $0$ and move by $\pm1$, they visit every level, and each of them only finitely many times. Thus, almost surely, every $\tau_n$ and $\bar\tau_n$ is a well-defined integer.

It is known (see~\cite[Theorem~1]{1989:doney}) that the differences $\bar\tau_{n+1} - \bar\tau_n$, for $n \geqslant 0$, are independent and follow the same law as the law of first passage of $(X_k)_k$ above 0, that is, the law of the random variable $\sigma = \min\{k \colon X_k > 0\}$.
Observing that $(X_{k+\tau_0})_k$ and $(Z_{k+\bar\tau_0})_k$ follow the same law then proves that the differences $\tau_{n+1} - \tau_n$, for $n \geqslant 0$, are independent and follow the same law as $\sigma$.

The probability generating function of $\sigma$, which can be found in~\cite[Section XIV.4]{1950:feller}, is given by
\[\bbE[z^{\sigma}] = \Phi(z) = \frac{1-\sqrt{1-4pqz^2}}{2qz}.\]
Since $\Phi$ admits derivatives of any order at $z = 1$, it follows that all moments of $\sigma$ are finite.
Furthermore, $\Phi$ is not a polynomial, which means that $\sigma$ is not bounded, and hence has a positive variance. The same properties hold therefore for the differences $(\tau_{n+1}-\tau_n)_n$

For each $w \in F(A)$, let $\Gamma_w$ denote the sub-graph of $\bbCay(F(A),A)$ which consists of the vertices of the form $w u$.
This is a tree, rooted at $w$, where each vertex has $\alpha-1$ children, except if $w = \varepsilon$, in which case $w$ has $\alpha$ children instead of $\alpha-1$.
In this latter case, $\Gamma_\varepsilon$ is just $\bbCay(F(A),A)$ itself.

Removing the ball $\{x \in F(A) \colon \|x\| \leqslant n\}$ splits the graph $\Gamma_\varepsilon$ into connected components $\Gamma_w$ where $\|w\| = n+1$.
Hence, consider some $y \in F(A)$  of length $\|y\| = n > 0$, and let $a$ be its rightmost letter if $n>0$. If $y=\varepsilon$, the letter $b$ below is unconstrained.
Conditionally to the event $Y_{\tau_n} = y$, the random walk $(Y_j)_{j > \tau_n}$ is confined to some tree $\Gamma_w$ where $\|w\| = n+1$.
In particular, both $Y_{\tau_n+1}$ and $Y_{\tau_{n+1}}$ coincide with the only element of $\Gamma_w$ of length at most $n+1$, \ie, with $w$.
This means that $Y_{\tau_{n+1}} = w = Y_{\tau_n+1}$ is a child of $y$ in $\Gamma_\varepsilon$, \ie, $w = yb$ for some letter $b \in \tilde{A} \setminus \{\bar a\}$.

Furthermore, since $\Gamma_\varepsilon$ is homogeneous, it admits an automorphism that exchanges any two trees $\Gamma_{yb}$ and $\Gamma_{yb'}$ ($b, b'\ne \bar a$) while fixing the rest of $\Gamma_\varepsilon$.
Consequently, the probability that $Y_{\tau_n+1} = yb$ ($b\ne \bar a$) does not depend on $b$, and hence must be $\frac{1}{\alpha-1}$ or $\frac{1}{\alpha}$ if $y=\varepsilon$.
As $Y_{\tau_{n+1}} = Y_{\tau_n+1}$, this proves that $(Y_{\tau_n})_n$ is a simple non-backtracking walk.
\end{proof}

Finally, the behavior of the random variables~$\tau_n$ mentioned in \Cref{lemma:tau n} can be controlled using the following classical result due to Berry~\cite[Theorem~1]{1941:berry} and, independently, to Esseen~\cite{1942:Esseen}.

\begin{theorem}[Berry-Esseen]\label{thm:esseen-berry}
Let~$(X_n)_{n \geqslant 1}$ be i.i.d. real-valued random variables with positive variance and finite third absolute moment and $Z_n = X_1+\cdots+X_n$. Let~$D$ be the distribution function of the standard normal distribution, \ie
\[D(x) = \frac{1}{\sqrt{2\pi}}\int_{-\infty}^x \exp\left(-\frac{t^2}2\right)\,\mathrm{d}t.\]
Then there exists~a constant $c > 0$ depending only on the variance and third moment of $X_1$ such that, for all~$n \geqslant 1$,
\[
\sup_{x\in\mathbb R} \left|{\mathbb P}\left(\frac{Z_n-\bbE [Z_n]}{\sqrt{\bbV [Z_n]}}\leq x\right) - D(x)\right| \leq \frac{c}{\sqrt{n}}.
\] 
\end{theorem}

The above results yield the following analogue of Theorem~\ref{thm:hebisch} for simple non-backtracking random walks.

\begin{theorem}\label{thm:non-backtracking}
Let $G$ be a finitely generated linear group and let $S$ be a finite generating set of $G$. Let $A$ be a finite set equipped with a bijection $\mat\colon A \to S$.

Let $(X_n)_n$ be a simple random non-backtracking walk on $\bbCay(F(A),A)$. If $G$ is infinite, there exists a constant $c > 0$ such that
\[\bbP[\mat(X_n) = g] \leqslant c n^{-1/4} \log^{3/4} n\]
for all integers~$n > 1$ and all~$g \in G$.

If, in addition, $G$ has exponential growth, then there exists a constant~$c > 0$ such that
\[\bbP[\mat(X_n) = g] \leqslant c \exp\left(- \frac{n^{1/3}} c\right)\]
for all integers~$n \geqslant 0$ and all~$g \in G$.
\end{theorem}

\begin{proof}
Let $(Y_n)$ be a simple random walk on $\bbCay(F(A),A)$, let~$\tau_n$ be the variables defined in \Cref{lemma:tau n}, and let~$Y'_n = Y_{\tau_n}$. Then, almost surely, $(Y'_n)_n$ is a simple non-backtracking random walk, and hence it follows the same law as $(X_n)_n$: $\bbP[X_n = x] = \bbP[Y'_n = x] = \bbP[Y_{\tau_n} = x]$ for every $x\in F(A)$. 

We observe that $(\mat(Y_n))_n$ is a random walk on $\bbCay(G,S)$. This random walk may not be simple (for each $s\in S\cup S\inv$, the probability that $\mat(Y_{n+1}) = \mat(Y_n)\,s$ is $|\tilde A \cap \mat\inv(s)| / |\tilde A|$) but it is symmetric and its support generates $G$, so it satisfies the hypotheses \Cref{thm:hebisch}. 

Since $G$ is an infinite linear group, its growth is either exponential or polynomial (\Cref{prop:tits}).

\smallskip

\noindent\underline{Suppose that $G$ has exponential growth.}
By construction, $\tau_n \geqslant n$ (since $Y_{\tau_n}$ is a reduced word of length $n$), so
\[\bbP[\mat(X_n) = g] = \bbP[\mat(Y_{\tau_n}) = g] \leqslant \sum_{k \geqslant n} \bbP[\mat(Y_k) = g].\] 
It follows, by \Cref{thm:hebisch}, that there exists a constant $c > 0$ such that, for each $n$
\[\bbP[\mat(X_n) = g] \leqslant c \sum_{k \geqslant n} \exp\left(-\frac{k^{1/3}}c\right).\]
The latter sum converges for all~$n \geqslant 1$, and we have
\[\bbP[\mat(X_n) = g] \leqslant c \exp\left(-\frac{n^{1/3}}c\right) + c \int_n^\infty \exp\left(-\frac{t^{1/3}}c\right) \, \mathrm{d}t.\]
It is an elementary verification that $- (6 c^3 + 6 c^2 t^{1/3} + 3 c t^{2/3}) \exp\left(-\frac{t^{1/3}}c\right)$ is a primitive of $\exp\left(-\frac{t^{1/3}}c\right)$. As a result, we have
$$\bbP[\mat(X_n) = g] \leqslant c (1 + 6 c^3 + 6 c^2 n^{1/3} + 3 c n^{2/3}) \exp\left(-\frac{n^{1/3}}c\right),$$
which proves that~$\bbP[\mat(X_n) = g] = \calO\left(\exp\left(-\frac{n^{1/3}}{2c}\right)\right)$. This concludes the proof of the statement in the exponential growth case.

\smallskip

\noindent\underline{Suppose $G$ has polynomial growth, of order $\delta > 0$.}
Observe that $\tau_n = \tau_1 + \sum_{i=1}^{n-1}(\tau_{i+1}-\tau_i)$.
By \Cref{lemma:tau n}, the $\tau_{i+1}-\tau_i$ are identically distributed, independent, with positive variance and finite third moment.
In particular, these random variables satisfy the hypotheses of \Cref{thm:esseen-berry}.
Let $c$ denote the constant given by that theorem and by $D$ the distribution of the standard Gaussian law (see the statement of \Cref{thm:esseen-berry}).

Note that $\tau_n = \tau_1 + \sum_{i=1}^{n-1}(\tau_{i+1}-\tau_i)$.
Since the $\tau_{i+1}-\tau_i$ are independent and identically distributed, it follows from the linearity of the expectation and the variance functions that both $\bbE(\tau_n)$ and $\Var(\tau_n)$ are $\Theta(n)$.

Let $\alpha = \lfloor\bbE(\tau_n) - \sqrt{\Var(\tau_n)\,\log n}\rfloor$, $\beta = \lceil\bbE(\tau_n) + \sqrt{\Var(\tau_n)\,\log n}\rceil$ and $\eta = \lceil \sqrt{c\, \delta\, (\beta-\alpha)\, \log n} \, \rceil$. 
Observe that $\alpha = \Theta(n)$ and $\eta = \Theta(n^{1/4}\log^{3/4} n)$ 
In order to bound $\bbP[\mat(X_n) = g] = \bbP[\mat(Y_{\tau_n}) = g]$ when $g\in G$, we cover the event $\{\mat(Y_{\tau_n}) = g\}$ by the following events:
$\{\tau_n \leqslant \alpha\}$;
$\{\tau_n \leqslant \beta\}$;
$\{\|\mat(Y_\alpha)\inv g\| \leqslant \eta\}$; and
$\{\|\mat(Y_\alpha)\inv g\| > \eta \text{, } \alpha < \tau_n < \beta \text{ and } \mat(Y_{\tau_n})=g\}$.
In particular, the probability $\bbP[\mat(X_n) = g]$ is at most equal to the sum of the probabilities of these four events.

First, $\bbP[\tau_n \leqslant \alpha]$ is bounded from above by $\bbP[\tau_n \leqslant \bbE(\tau_n) - \sqrt{\Var(\tau_n)\,\log n}]$ and hence, thanks to \Cref{thm:esseen-berry}, by $\frac c{\sqrt n} + D(-\sqrt{\log n})$.
An elementary computation shows that 
\[D(-\sqrt{\log n}) =
\frac{1}{\sqrt{2\pi}}\int_{-\infty}^{-\sqrt{\log n}} \exp\left(-\frac{t^2}2\right)\,dt
= \frac{1}{\sqrt{2\pi}}\int_{+\infty}^{\frac12\log n}\frac{e^{-s}}{-\sqrt {2s}}\, ds
\leqslant \frac1{\sqrt{2\pi\log n}} \int_{\frac12\log n}^{\infty} e^{-s}\, ds
= \frac1{\sqrt{2\pi \, n \log n}}.\]
As a result, $\bbP[\tau_n \leqslant \alpha] = \calO(n^{-1/2})$.

Second, similarly, $\bbP[\tau_n > \beta] = 1 - \bbP[\tau_n\leqslant \beta] \leqslant \frac c{\sqrt n} + 1 - D(\sqrt{\log n}) = \frac c{\sqrt n} + D(-\sqrt{\log n}) = \calO(n^{-1/2})$.

Third, let us consider the event $\|\mat(Y_\alpha)\inv g\| \leqslant \eta$, \ie, $\mat(Y_\alpha) \in B_S(g,\eta)$.
By \Cref{thm:hebisch}, for each $g' \in B_S(g,\eta)$, we have $\bbP[\mat(Y_\alpha) = g'] \leqslant c\,\alpha^{-\delta/2} = \calO(n^{-\delta/2})$.
Since $|B_S(g,\eta)| = \Theta(\eta^\delta) = \Theta(n^{\delta/4}\log^{3\delta/4} n)$, it follows from the union bound that $\bbP[\|\mat(Y_\alpha)\inv g\| \leqslant \eta] = \calO(n^{-\delta/4}\log^{3\delta/4} n)$, and hence $\bbP[\|\mat(Y_\alpha)\inv g\| \leqslant \eta] = \calO(n^{-1/4}\log^{3/4} n)$ since $\delta \geqslant 1$.

Finally, we cover the event $\{\|\mat(Y_\alpha)\inv g\| > \eta \text{, }\alpha < \tau_n \leqslant \beta \text{ and } \mat(Y_{\tau_n})=g\}$ by the events $\{\mat(Y_\alpha)\inv g = y$ and $\mat(Y_{\alpha+t})=g\}$, where $\|y\| > \eta$ and $1 \leqslant t < \beta - \alpha$:
we simply require that $Y_{\alpha + t} = g$ for some $t < \beta - \alpha$, even if $\alpha + t$ does not coincide with $\tau_n$.
Let $(Z_n)_{n\geqslant 0}$ be the random walk given by $Z_n = Y_{\alpha}\inv\,Y_{n+\alpha}$.
Since random walk is independent from $Y_\alpha$, we have
\[\bbP[\mat(Y_\alpha)\inv g = y \text{ and } \mat(Y_{\alpha+t})=g] =
\bbP[\mat(Y_\alpha)\inv g = y \text{ and } \mat(Z_t)=y] =
\bbP[\mat(Y_\alpha)\inv g = y] \, \bbP[\mat(Z_t)=y].\]
Moreover, by \Cref{thm:hebisch}, we have
\[\bbP[\mat(Z_t)=y] \leqslant c \, t^{-\delta/2} \exp\left(-\frac{\|y\|^2}{ct}\right)
\leqslant f(t),\]
where we set $f(t) = c\,t^{-1/2}\exp\left(-\frac{\eta^2}{ct}\right)$.

The derivative of $f$ is $f'(t) = (2\eta^2 - c t) f(t) / (2 t^2)$.
When $n$ is large enough, this is positive on the interval $[1,\beta-\alpha]$ since $c t \leqslant c (\beta - \alpha) = \Theta(\eta^2 / \log n)$, and thus $f(t) \leqslant f(\beta-\alpha)$.
Therefore, by union bound, we have
\begin{align*}
\bbP[\, \|\mat(Y_\alpha)\inv g\| > \eta \text{, } \alpha < \tau_n < \alpha + \eta \text{ and } \mat(Y_{\tau_n})=g \,]
& \leqslant \sum_{y \colon \|y\| > \eta} \sum_{t=1\vphantom{\|y\|}}^{\beta-\alpha} \bbP[\mat(Y_\alpha)\inv g = y] \, \bbP[\mat(Z_t)=y] \\
& \leqslant \sum_{y \colon \|y\| > \eta} \bbP[\mat(Y_\alpha)\inv g = y] \, (\beta-\alpha) \, f(\beta-\alpha) \\
& \leqslant (\beta-\alpha) \, f(\beta-\alpha).\end{align*}
And $(\beta-\alpha) f(\beta-\alpha) = c (\beta-\alpha)^{1/2}\exp\left(-\frac{\eta^2}{c(\beta-\alpha)}\right) \leqslant c (\beta - \alpha)^{1/2} \exp(- \delta \log n) = \calO(n^{1/4-\delta}\log^{1/4}n) = \calO(n^{-1/2})$, using $\delta \geqslant 1$.

Combining the four bounds above gives $\bbP[\mat(X_n)=g] \leqslant \calO(n^{-1/2}) + \calO(n^{-1/4}\log^{3/4}n) = \calO(n^{-1/4}\log^{3/4}n)$, matching the statement of the theorem.
\end{proof}

Theorems~\ref{thm:hebisch} and~\ref{thm:non-backtracking} yield the following corollary. 

\begin{corollary}\label{cor:rare-events}
Let $A$, $\mat$, $G$ be as in \Cref{thm:non-backtracking}.
For each pair of integers $(m,n)$ such that $m\leqslant n$, let $\bbP_{m,n}$ be the probability that a random word of length $n$ (resp. a random reduced word of length $n$) $w$ satisfies $\mat(w)_{\equiv m} = \Id$.
Let also $\bbP_m = \sup\{\bbP_{m,n} : n \ge m\}$.

 If $G$ is infinite, there exists a constant~$c > 0$ such that, for all large enough $m$, $\bbP_m \leqslant c\ \log^{-1/4}m\ \log^{3/4}\log m$.
If, in addition, $G$ has exponential growth, then $\bbP_m \leqslant c \exp\left(-\frac{\log^{1/3}m}c\right)$.
\end{corollary}

\begin{proof}
Let us start with a preliminary observation.
Let $\sfB=\max\{\ell(\mat(a)):a\in\tilde A\}$ and 
$p= \left\lfloor \frac{\log m-1+\log d}{\sfB+\log d} \right\rfloor$. If $m$ is large enough, we have $n \geqslant m > p \geqslant 1$.
By Proposition~\ref{prop: basic facts}~\eqref{item: length of mat(w)}, every coefficient of a matrix of the form $\mat(u)$, with $|u|\leqslant p$, has length at most
\[ 1 -\log d+p\,(\sfB + \log d) \leqslant \log m.\]
Hence each coefficient of $\mat(u)$ has absolute value strictly smaller than $m/2$. Therefore reduction modulo $m$, with residues chosen of least absolute value, is injective on the set of matrices $\mat(u)$ with $|u|\leqslant p$.
It follows that, if $w$ is a word of length $n$ and $w = xy$ with $|x| = p$, then $\mat(w)_{\equiv m} = \Id$ if and only if $\mat(x)_{\equiv m} = \mat(y)_{\equiv m}\inv$, if and only if $\mat(x) = \mat(y)_{\equiv m}\inv$.

We then prove~\Cref{cor:rare-events} in the general case, \ie, the case where input words are not assumed to be reduced.
Let $W$ be a random variable uniformly distributed over words of length $n$, let $X$ be the length $p$ prefix of $W$ and let $Y$ be its suffix of length $n-p$.
Note that $X$ is independent of $Y$, and $X$ is a uniform random word of length $p$.
This yields
$$\bbP[\mat(W)_{\equiv m} = \Id] = \sum_{M} \bbP[\mat(X) = M]\ \bbP[\mat(Y)_{\equiv m} = M\inv] \leqslant \max_M \bbP[\max(X) = M],$$
where the sum and the maximum run over all matrices $M\in G_{\equiv m}$. 

By Theorem~\ref{thm:hebisch}, there exists a constant $c'>0$ such that, for every matrix $M$ in $G$, we have $\bbP[\mat(X) = M] \leqslant c'\,p^{-1/2}$ if $G$ is infinite, and even $\bbP[\mat(X) = M] \leqslant c'\exp\left(-\frac{p^{1/3}}{c'}\right)$ if $G$ has exponential growth.
Since $p = \Theta(\log m)$, there exists a constant $c_1 > 0$ such that $c_1\log m \leqslant p$ for all $m$ large enough.
In case $G$ is infinite, it follows that
$$\bbP[\mat(W)_{\equiv m} = \Id] \leqslant c'\,p^{-1/2} \leqslant c'c_1^{-1/2}\,\log^{-1/2}m$$
if $m$ is large enough.
In case $G$ has exponential growth, we directly observe that
$$\bbP[\mat(W)_{\equiv m} = \Id] \leqslant c'\exp\left(-\frac{p^{1/3}}{c'}\right) \leqslant c'\exp\left(-\frac{c_1^{1/3}}{c'}\, \log^{1/3}m\right) \leqslant c \exp\left(-\frac{\log^{1/3}m}c\right),$$
where $c = \max(c',\frac{c'}{c_1^{1/3}})$.
This completes the proof in the general (unreduced) case.

\smallskip
We now turn to the case where words are chosen uniformly at random among the reduced words of length $n$ on $\tilde A$.
Let $W$, $X$ and $Y$ as above, except that $W$ is now uniformly distributed over reduced words of length $n$.
Unlike in the general case, $X$ and $Y$ are not independent variables:
since $W$ is reduced, so are $X$ and $Y$, and the last letter of $X$ is constrained to not be the inverse of the first letter of $Y$.

Let $\hat{X}$ denote an unconstrained reduced word of length $p$ chosen uniformly at random. Let $x$ denote the rightmost letter of $\hat{X}$, and let $y$ denote the leftmost letter of $Y$.
For each fixed letter $a \in \tilde{A}$ and each set $\calE$ of words, since $x$ is uniformly distributed among $\tilde{A}$, we have
\[\bbP[\hat{X} \in \calE \mid x \neq a] \enspace\leqslant\enspace \frac{\bbP[\hat{X} \in \calE]}{\bbP[x \neq a]} \enspace=\enspace \frac\alpha{\alpha-1}\, \bbP[\hat{X} \in \calE],\]
where $\alpha = |\tilde{A}|$
As in the non-reduced case, but distinguishing all values that $y$ may take, we observe that
\begin{align*}
\bbP[\mat(W)_{\equiv m} = \Id] 
& = \sum_{a,M} \bbP[\mat(\hat{X}) = M \mid x \neq a] \ \bbP[\mat(Y)_{\equiv m} = M\inv \text{ and } y = \bar a] \\
& \leqslant \frac{\alpha}{\alpha-1}\, \sum_{a,M} \bbP[\mat(\hat{X}) = M] \ \bbP[\mat(Y)_{\equiv m} = M\inv \text{ and } y = \bar{a}] \\
& \leqslant \frac{\alpha}{\alpha-1}\, \max_M \bbP[\mat(\hat{X}) = M],
\end{align*}
where the sum and the maximum run over all matrices $M \in G_{\equiv m}$ and letters $a \in \tilde{A}$.

By \Cref{thm:non-backtracking}, there exists a constant $c' > 0$ such that, for every matrix $M$ in $G$, we have $\bbP[\mat(\hat{X}) = M] \leqslant c'\,p^{-1/4}\log^{3/4} p$ if $G$ is infinite, and even $\bbP[\mat(\hat{X}) = M] \leqslant c'\exp\left(-\frac{p^{1/3}}{c'}\right)$ if $G$ has exponential growth.
Since $p = \Theta(\log m)$, there exist constants $c_1 > 0$ and $c_2 > 0$ such that $c_1\log m \leqslant p \leqslant c_2 \log m$ for all $m$ large enough.
In case $G$ is infinite, it follows that 
$$\bbP[\mat(W)_{\equiv m} = \Id] \leqslant \frac{\alpha}{\alpha-1} c'\,p^{-1/4}\log^{3/4} p \leqslant
\frac{\alpha}{\alpha-1}\,c'\,(c_1\log m)^{-1/4}\big(\log c_2 + \log\log m\big)^{3/4} \leqslant c\,\log^{-1/4}m\,\log^{3/4}\log m$$
for a suitable constant $c > 0$ and all $m$ large enough.
In case $G$ has exponential growth, we observe, as in the non-reduced case, that
$$\bbP[\mat(W)_{\equiv m} = \Id] \enspace\leqslant\enspace c'\exp\left(-\frac{p^{1/3}}{c'}\right) \enspace\leqslant\enspace c'\exp\left(-\frac{c_1^{1/3}}{c'}\, \log^{1/3}m\right) \enspace\leqslant\enspace c\  \exp\left(-\frac{\log^{1/3}m}c\right)$$
for a suitable constant $c > 0$.
This concludes the proof of the corollary.
\end{proof}

%%%%%%%%%%%%%%%%%%%%
\section{Algorithms solving the Word Problem for matrices over $\mathbb{Z}$}\label{section:algorithms}

In the following, we give several algorithms to solve the Exact Computation and the Word Problem. We start with standard algorithms, for which we compute the worst-case complexity. We show in particular that a standard divide-and-conquer approach yields a $\calO(n\,\log^2n)$ worst-case complexity. In \Cref{sec: a better algorithm}, we then present our own algorithm, and show that it achieves the same worst-case complexity while having linear average-case complexity.

We conclude the section with a short discussion of particular situations: when $G$ is finite or generated by upper-triangular matrices, in which case the divide-and-conquer algorithm (and ours as well) has linear worst-case complexity; and when $G$ is free.
Extensions of our algorithms to matrices over $\bbQ$, and over the ring of integers of a number field, are discussed in \Cref{sec: larger rings}.

Throughout this section, $A$ is a fixed non-empty finite set, $\tilde A = A \sqcup \bar A$ is the disjoint union of $A$ and a set $\bar A$ of formal inverses, $\mat\colon A \to \GL_d(\bbZ)$ is a map extended to a morphism on $\tilde A^*$ or $F(A)$ depending on the context, $\sfB$ is the maximum length of the matrices in $\mat(\tilde A)$ and $G = \langle \mat(A)\rangle$ is the subgroup of $\GL_d(\bbZ)$ generated by $\mat(A)$.

%%%%%%%%%%%%%%%%%%%%
\subsection{Standard algorithms for the Exact Computation Problem}\label{sec: first algos}

The algorithms described in this section solve $\EC_A$. We observe that these yield algorithms solving~$\WP_A$: having computed $\mat(w)$, we only need to examine the~$d^2$ coefficients of~$\mat(w)$ to decide whether $\mat(w) = \Id$, and this is done in constant time (see \Cref{prop: basic facts} \eqref{item: comparing integers}).

\paragraph{Naive algorithm~$\NGr_A$.}
The naive algorithm to compute~$\mat(w)$ consists in reading the word~$w$ from left to right, one letter at a time, and performing the corresponding~$n-1$ matrix multiplications, by following algorithm~$\NGr_A$ below.

\medskip

\begin{algorithm}[H]
\caption{\bf Algorithm~$\NGr_A$}
\DontPrintSemicolon
\SetKwInOut{Input}{Input}
\SetKwInOut{Output}{Output}

\Input{a word~$w = w_1 w_2 \cdots w_n$ of~$n$ elements on alphabet~$\tilde A$}
\Output{$\mat(w)$}

\lIf{$n=0$}{\Return~$\Id$}
\lElse{\Return~$\NGr_A(w_1 w_2 \cdots w_{n-1})\times \mat(w_n)$}
\end{algorithm}

\medskip

The right factor in each of these operations is a matrix in~$\mat(\tilde A)$, a constant set which is independent of~$n$. As a result (by \Cref{prop: basic facts}), step $i$ in Algorithm $\NGr_A$ takes time at most $c\,i\,\sfB$ (for some constant $c>0$), and the worst-case complexity of the algorithm is therefore $\calO(n^2)$.

\paragraph{Divide-and-conquer algorithm~$\DC_A$.}
This upper bound can be significantly improved by using a classical \emph{divide and conquer} strategy, which is the core idea behind the following algorithm.

\medskip

\begin{algorithm}[H]
\caption{\bf Algorithm~$\DC_A$}
\DontPrintSemicolon
\SetKwInOut{Input}{Input}
\SetKwInOut{Output}{Output}

\Input{a word~$w$ of~$n$ elements on alphabet~$\tilde A$}
\Output{$\mat(w)$}

\lIf{$n=0$}{\Return~$\Id$}
\lElseIf{$n=1$}{\Return $\mat(a)$, where $a$ is the unique letter of~$w$}
\lElse{$m \gets \lfloor n/2 \rfloor$\;
\Return~$\DC_A(w_1 w_2 \cdots w_m)\times\DC_A(w_{m+1} w_{m+2} \cdots w_n)$}
\end{algorithm}

\medskip

The complexity analysis of this simple algorithm is based on the celebrated \emph{Master Theorem} (see, \textit{e.g.}, \cite[Theorem 4.1]{cormen2022introduction}) which we will use several times. We cite here only the specific instance of the Master Theorem that we will use.

\begin{proposition}\label{prop: master theorem}
Let~$T \colon \bbN \to \bbN$ and~$f \colon \bbN \to \bbN$ be functions satisfying the equation~$T(n) = T(\lfloor n/2 \rfloor) + T(\lceil n/2 \rceil) + f(n)$ for all~$n \geqslant 2$:
\begin{itemize}
\item if~$f(n) = \calO(n \log^{-2}n)$, then~$T(n) = \calO(n)$; 
\item if~$f(n) = \calO(n \log^kn)$ for some integer~$k \geqslant 0$, then~$T(n) = \calO(n\,\log^{k+1}n)$.
\end{itemize}
\end{proposition} 

The following evaluation of the worst-case complexity of $\DC_A$ was observed by Ol'shanskii and Shpilrain \cite[Proposition 2]{2025:OlshShp}.

\begin{corollary}\label{cor: complexity DC}
Algorithm~$\DC_A$ solves Problem~$\EC_A$, with worst-case complexity~$\calO(n\log^2n)$.
\end{corollary}

\begin{proof}
The correctness of Algorithm~$\DC_A$ follows from an immediate induction on~$n$.
Then, Proposition~\ref{prop: basic facts} shows that both matrices~$\mat(w_1 w_2 \cdots w_m)$ and~$\mat(w_{m+1} w_{m+2} \cdots w_n)$ have entries of length at most~$\sfB (n+1) / 2$, so that the complexity~$C(n)$ of~$\DC_A$ satisfies the equation
\begin{equation*}%\label{eq:divide conquer}
C(n) = C\left(\lfloor n/2\rfloor\right) + C\left(\lceil n/2\rceil\right) + \calO(n\log n)
\end{equation*}
for~$n \geqslant 2$.
Proposition~\ref{prop: master theorem} then yields that~$C(n)$ is in~$\calO(n \log^2n)$.
\end{proof}

\begin{corollary}\label{cor: complexity DC finite}
If the group $G$ generated by $\mat(A)$ is finite, then Algorithm~$\DC_A$ solves Problem~$\EC_A$ with worst-case complexity~$\calO(n)$.
\end{corollary}

\begin{proof}
Because $G$ is finite, the lengths of the matrices $\mat(w)$ ($w\in \tilde A^*$) are uniformly bounded, independently of $|w|$. It follows that the complexity~$C(n)$ of~$\DC_A$ satisfies the equation
\begin{equation*}%\label{eq:divide conquer}
C(n) = C\left(\lfloor n/2\rfloor\right) + C\left(\lceil n/2\rceil\right) + \calO(1)
\end{equation*}
for~$n \geqslant 2$.
Proposition~\ref{prop: master theorem} now yields $C(n) = \calO(n)$.
\end{proof}

%

%%%%%%%%%%%%%%%%%%%%
\subsection{An algorithm with linear average-case complexity for the Word Problem}\label{sec: a better algorithm}

Our new algorithm is aimed at providing a better average complexity without sacrificing the worst-case complexity of algorithm~$\DC_A$.

For each integer $s\geqslant 0$, let $P_s = \mat(\tilde A^{2^s})$, \ie, $P_s$ is the set of matrices of the form $\mat(u)$ where $|u| = 2^s$. Obviously, $|P_s| \leqslant |\tilde A|^{2^s}$. We also let $\hat P_s = \sum_{t\leqslant s}|P_t|$.

The general scheme of our algorithm is the following. We are given an input word $w \in \tilde A^*$.
\begin{enumerate}[(i)]
\item\label{item: compute Ps} We compute, in time $\calO(n)$, the sets $P_s$ for $s = 0, 1, \dots$, up to a value $s = \sigma$ maximal for the following properties: $2^\sigma \leqslant \frac{2n^{1/2}}{\log n}$ and $\hat P_{\sigma-1} \leqslant n^{1/4}$.
\item\label{item: compute Ps mod m} Let $\tau = \lceil\frac{2^\sigma}\sigma\rceil$ and $m = 2^\tau$. We compute the sets $(P_s)_{\equiv m}$, $s\leqslant \sigma$, in time $\calO(n)$.
\item\label{item: compute mat(w) mod m} We use Steps~\eqref{item: compute Ps} and~\eqref{item: compute Ps mod m} to compute $\mat(w)_{\equiv m}$ in time $\calO(n)$.
\item\label{item: decide the WP} If $\mat(w)_{\equiv m} \ne \Id$, we conclude that $\mat(w) \ne \Id$. If instead $\mat(w)_{\equiv m} = \Id$, we run $\DC_A$ on input $w$ to compute $\mat(w)$ and decide whether it is equal to $\Id$.
\end{enumerate}
It is clear that this algorithm decides $\WP_A$. The proof that it does so in linear average time consists in explaining the algorithms in Steps~\eqref{item: compute Ps}, \eqref{item: compute Ps mod m} and \eqref{item: compute mat(w) mod m}, including verifying that they run in (worst-case) time $\calO(n)$, and in establishing that the probability that $\mat(w)_{\equiv m} = \Id$ is $\calO(\log^{-2}n)$, in order to offset the $\log^2n$ factor in \Cref{cor: complexity DC}. This will also establish that our algorithm has worst-case complexity $\calO(n\,\log^2n)$.

\paragraph{Step \eqref{item: compute Ps}}
We compute the integers $\sigma$ and $\hat P_\sigma$, a map $\mu$ from $[1,\hat P_\sigma]$ to $\bigcup_{s=0}^\sigma P_s$ and a map $\sfL$ in such a way that 
\begin{itemize}
\item for each $0 \leqslant s \leqslant \sigma$, $\mu$ is a bijection from the interval $[\hat P_{s-1}+1, \hat P_s]$ to $P_s$ (letting $\hat P_{-1} = 0$);
\item for each $0 \leqslant s < \sigma$ and for any $k_1, k_2 \in [\hat P_{s-1}+1, \hat P_s]$ ($s < \sigma$), we have $\sfL(k_1,k_2) = k$ if and only if $k \in[\hat P_s+1, \hat P_{s+1}]$ and $\mu(k) = \mu(k_1)\mu(k_2)$.
\end{itemize}

In order to perform this computation, we will use an efficient data structure $\sfT$ which assigns a label $\sflabel_\sfT(M) \in \bbN$ to matrices $M\in \GL_d(\bbZ)$ and allows the following operations: given $M\in \GL_d(\bbZ)$,
\begin{enumerate}[(a)]
 \item check whether $M$ is in the domain of $\sfT$;
 \item if it is, output $\sflabel_\sfT(M)$;
 \item if it is not, add $M$ to $\sfT$ by letting $\sflabel_\sfT(M)$ be the least integer not in the range of $\sflabel_\sfT$.
\end{enumerate}
The red-black trees introduced by Guibas and Sedgewick~\cite{1987:guibas} (see also \cite[Chapter 13]{cormen2022introduction}) provide such a data structure. With this data structure, Operations (a), (b) and (c) above require $\calO(\log D)$ matrix comparisons, where $D$ is the number of matrices already labeled.

The construction of $\sigma$, $\mu$ and $\sfL$ is summarized in Algorithm $\Bounds_n$ below. The first part of the algorithm (lines \ref{tn: begin initial phase} to \ref{tn: end initial phase}) computes $\hat P_0 = P_0$ and defines $\mu$ on the interval $[1,\hat P_0]$. Recall that $P_0$ is the set of matrices of the form $\mat(a)$ ($a\in \tilde A$). A data structure $\sfT$ as described above is used to account for the fact that some of these matrices may coincide.

The next step (the loop of lines~\ref{tn:while-loop} to~\ref{tn:while-loop:end}) handles the successive values of $s > 0$, using a fresh version of $\sfT$. In substance, it computes $P_s = \mat(\tilde A^{2^s})$ using the fact that every matrix in that set is the product of two matrices of $P_{s-1}$.

\medskip

\begin{algorithm}[H]
\caption{\bf Algorithm~$\Bounds_n$}
\DontPrintSemicolon
\SetKwInOut{Input}{Input}
\SetKwInOut{Output}{Output}

$\mathsf{T} \gets$ empty map\;\label{tn: begin initial phase}
$\hat P \gets 0$\;
\For{$a \in \tilde A$}{
{}$\ell \gets \sflabel_\mathsf{T}(\mat(a))$\;
{}$\mu(\ell) \gets \mat(a)$\;
{}$\hat P \gets \max\{\hat P,\ell\}$\;
}
$(P,Q,s) \leftarrow (0,\hat P,0)$\;\label{tn: end initial phase}
\While{$\hat P \leqslant n^{1/4}$ and~$2^s \leqslant n^{1/2}/\log n$\label{tn:while-loop}}{
$\mathsf{T} \gets$ empty map\;
\For{$j=P+1 \ldots Q$}{
 \For{$k=P+1 \ldots Q$}{
 {}$M \gets \mu(j)\, \mu(k)$\;
 {}$\ell \gets Q+\sflabel_\mathsf{T}(M)$\;
 {}$\mu(\ell) \gets M$\;
 {}$\sfL(j,k) \gets \ell$\;
 {}$\hat P \gets \max\{\hat P,\ell\}$\;
}}
$(P,Q,s) \gets (Q+1,\hat P,s+1)$\;\label{tn:while-loop:end}
}
\Return~$(\mu,\sfL,s)$
\end{algorithm}
\medskip

The algorithm stops when $s$ is maximal such that $\hat P_{s-1} \leqslant n^{1/4}$ and $2^s \leqslant \frac{2n^{1/2}}{\log n}$. The integers $\sigma$ and $\hat P_\sigma$ are the final values of variables $s$ and $Q$.

\begin{proposition}\label{pro:time-of-tn}
Algorithm~$\Bounds_n$ runs in time~$\calO(n)$ in the worst case.

Moreover, the integer~$\sigma$ satisfies the lower bound $2^\sigma = \Omega(\log n)$ in general.
More precisely, if the group $G$ has exponential growth, then $2^\sigma = \Theta(\log n)$;
if $G$ has polynomial growth of order~$\delta \geqslant 1$, then $2^\sigma = \Omega(n^{1/(4\delta)})$;
and if $G$ is finite, then $2^\sigma = \Theta(\frac{n^{1/2}}{\log n})$.
\end{proposition}

\begin{proof}
Note that, for each $s$, the set $P_s$ is of size $|P_s| = |\mat(\tilde A^{2^s})| \leqslant |\tilde A|^{2^s}$.
Each iteration of the \textbf{while} loop (lines~\ref{tn:while-loop} to~\ref{tn:while-loop:end}) increments the value of $s$ by 1.
During each of these iterations, one needs to compute the $|P_s|^2$ products of two matrices from $P_s$, and to look up (and, if new, insert) each of them in the data structure~$\sfT$, at a cost of $\calO(\log|P_{s+1}|)$ matrix comparisons per product.

Thus, the total number of products we need to compute is $\sum_{s \leqslant \sigma-1} |P_s|^2 \leqslant \hat P^2_{\sigma-1} \leqslant n^{1/2}$.
Since the \textbf{while} loop only continues while $\hat P \leqslant n^{1/4}$, we have $|P_{s+1}| \leqslant \hat P_{\sigma-1} \leqslant n^{1/4}$ for every $s < \sigma$, and hence $\log|P_{s+1}| = \calO(\log n)$.
The total number of matrix comparisons is therefore
\[\calO\Big(\log n \ \sum_{s\leqslant \sigma-1} |P_{s}|^2\Big) = \calO(n^{1/2}\log n).\]

Moreover, \Cref{prop: basic facts}~\eqref{item: length of mat(w)} shows that the matrices in $P_s = \mat(\tilde A^{2^s})$ have length at most $2^s\,(\sfB+\log d)$, and thus $\calO(2^s)$, where $\sfB$ is the maximum length of a matrix in $\mat(\tilde A)$.
In particular, each matrix comparison takes time $\calO(2^s) = \calO(\frac{n^{1/2}}{\log n})$ and each product takes time $\calO(2^s \log(2^s)) = \calO(n^{1/2})$.
Thus the total (worst-case) complexity of Algorithm $\Bounds_n$ is $\calO(n)$ as announced.

By construction, either $\hat P_\sigma > n^{1/4}$, or $2^\sigma > \frac{n^{1/2}}{\log n}$. As we noted before, $\hat P_\sigma \leqslant \sum_{s\leqslant \sigma}|\tilde A|^{2^s} < |\tilde A|^{2^\sigma+1}$. Therefore, if $\hat P_\sigma > n^{1/4}$, then $2^{\sigma}+1 > \frac14\,\frac{\log n}{\log|\tilde A|}$, and $2^\sigma = \Omega(\log n)$.

Suppose, moreover, that the group $G$ has exponential growth. Then there is a constant $C>1$ such that, for every $s$, $|P_s| \geqslant C^{2^s}$. In particular, $\hat P_{\sigma-1} \geqslant \sum_{s \leqslant \sigma-1}C^{2^s} > C^{2^{\sigma-1}}$. Since $P_{\sigma-1} \leqslant n^{1/4}$, it follows that $C^{2^{\sigma-1}} \leqslant n^{1/4}$. This in turn yields $2^\sigma = \calO(\log n)$, and hence $2^\sigma = \Theta(\log n)$.

Now, suppose instead that the group $G$ has polynomial growth of order $\delta\geqslant 1$. Then there is a constant $C > 0$ such that $\hat P_\sigma \leqslant C\sum_{s\leqslant \sigma}2^{s\delta} < C\,2^{\delta(\sigma+1)}$. If $\hat P_\sigma > n^{1/4}$, this yields $2^\sigma = \Omega(n^{1/(4\delta)})$, as announced for groups of polynomial growth.

Finally, if $G$ is finite, then $|P_s| \leqslant |G|$ for each $s$, and hence $\hat P_s \leqslant s|G|$. Now, for $n$ large enough, $\hat P_s > n^{1/4}$ implies that $s > \frac{n^{1/4}}{|G|}$ and hence $2^s > \frac{n^{1/2}}{\log n}$. Therefore we have $ \frac{n^{1/2}}{\log n} < 2^\sigma \leqslant 2\,\frac{n^{1/2}}{\log n}$, that is $2^\sigma = \Theta(\frac{n^{1/2}}{\log n})$.
\end{proof}

\paragraph{Step \eqref{item: compute Ps mod m}}
Let $\tau = \lceil\frac{2^\sigma}\sigma\rceil$ and $m = 2^\tau$.

\begin{lemma}\label{lem: compute some mat(w) mod m}
Computing~$\mu(k)_{\equiv m}$ for each integer $k \in [1,\hat P_\sigma]$ can be done in time~$\calO(n)$ in the worst case.
\end{lemma}

\begin{proof}
By construction of $\sigma$ (\Cref{item: compute Ps}), $2^\sigma \leqslant \frac{2n^{1/2}}{\log n}$, and $\tau = \lceil 2^\sigma/\sigma\rceil \leqslant 2^\sigma \leqslant \frac{2n^{1/2}}{\log n}$ and the representation of the modulus $m = 2^\tau$ has $\ell(m) = \tau+2$ bits.

Let $k \leqslant \hat P_\sigma$, then $\mu(k) = \mat(u)$ for some $u \in \tilde A^*$ with $|u| \leqslant 2^\sigma$. Writing $\sfB = \max_{a\in\tilde A}\ell(\mat(a))$, \Cref{prop: basic facts}~\eqref{item: length of mat(w)} gives $\ell(\mu(k)) \leqslant |u|\,(\sfB+\log d) \leqslant 2^\sigma(\sfB+\log d) = \calO(2^\sigma) = \calO(\frac{n^{1/2}}{\log n})$.
By \Cref{prop: basic facts}~\eqref{item: residue mod m}, computing $\mu(k)_{\equiv m}$ takes time $\calO(n^{1/2})$.
Since $\hat P_\sigma \leqslant \hat P_{\sigma-1} + P_{\sigma-1}^2 \leqslant n^{1/4}+n^{1/2}$, the total cost over all $k\leqslant\hat P_\sigma$ is $\calO(n^{1/2})\times\calO(n^{1/2})=\calO(n)$, as announced.
\end{proof}

\paragraph{Step \eqref{item: compute mat(w) mod m}}
We now efficiently compute $\mat(w)_{\equiv m}$ for any word $w \in \tilde A^*$. The simplest case arises if $w$ has length a power of 2, say $|w| = 2^s$ for some $s\leqslant \sigma$. In that case, we only need to retrieve the integer $k \in [\hat P_{s-1}+1,\hat P_s]$ such that $\mu(k) = \mat(w)$. This is done using the following divide-and-conquer algorithm. We consider that we have constant time access to the restriction of the map $\mu$ to the interval $[1,\hat P_1]$ (that is: to the bijection between that interval and the set $\mat(\tilde A)$.

\medskip

\begin{algorithm}[H]
\caption{\bf Algorithm~$\Lab$}
\DontPrintSemicolon
\SetKwInOut{Input}{Input}
\SetKwInOut{Output}{Output}

\Input{a map $\sfL$ and a word~$w = a_1 a_2 \cdots a_{2^s}$ on alphabet~$\tilde A$, where~$s \leqslant \sigma$}
\Output{an integer $k \in [\hat P_{s-1}+1, \hat P_s]$ such that $\mu(k) = \mat(w)$}
\lIf{$s=0$}{\Return~$k$ such that $\mu(k) = \mat(a_1)$}
\Else{
$k_1 = \Lab(a_1 a_2 \cdots a_{2^{s-1}})$\;
$k_2 = \Lab(a_{2^{s-1}+1} a_{2^{s-1}+2} \cdots a_{2^s})$\;
\Return~$\sfL(k_1,k_2)$
}
\end{algorithm}

\medskip

A direct application of \Cref{prop: master theorem} shows that the complexity of Algorithm $\Lab$ is $\calO(2^s)$.
We can now give an efficient algorithm $\ModC$ to compute~$\mat(w)_{\equiv m}$ for words~$w$ of arbitrary length.

\medskip

\begin{algorithm}[H]
\caption{\bf Algorithm~$\ModC$}
\DontPrintSemicolon
\SetKwInOut{Input}{Input}
\SetKwInOut{Output}{Output}

\Input{a word~$w = a_1 a_2 \cdots a_n$ on alphabet~$\tilde A$}
\Output{$\mat(w)_{\equiv m}$}
Run Algorithm $\Bounds_n$ to compute $(\mu, \sfL, \sigma)$\;
$m \gets2^{\lceil 2^\sigma/\sigma\rceil}$\;
Return $\ModC^*(\mu, \sfL, \sigma, m, w)$
\end{algorithm}

\begin{algorithm}[H]
\caption{\bf Algorithm~$\ModC\smash{^*}$}
\DontPrintSemicolon
\SetKwInOut{Input}{Input}
\SetKwInOut{Output}{Output}

\Input{maps $\mu$ and $\sfL$, integers $\sigma$ and $m$, and a word~$w = a_1 a_2 \cdots a_n$ on alphabet~$\tilde A$}
\Output{$\mat(w)_{\equiv m}$}
\lIf{$n=0$}{\Return~$\Id$}
\Else{
$t \gets \min\{\sigma, \lfloor \log n \rfloor\}$\;
$k \gets \Lab(\sfL, w_1 w_2 \cdots w_{2^t})$\label{fcm:lambda}\;
$A \gets \mu(k)_{\equiv m}$\label{fcm:a}\;
$B \gets \ModC^*(\mu, \sfL, \sigma, m, w_{2^t+1} w_{2^t+2} \cdots w_n)$\;
\Return~$(A B)_{\equiv m}$
}
\end{algorithm}

\begin{proposition}\label{prop: complexity of calM}
Algorithm~$\ModC$ runs in time~$\calO(n)$ in the worst case.
\end{proposition}

\begin{proof}
We already know that $\Bounds_n$ runs in time $\calO(n)$. Moreover, $\log m = \lceil 2^\sigma/\sigma\rceil$, so that $\log m\, \log\log m = \calO(2^\sigma)$.

Each recursive call of $\ModC^*$ consists in running $\Lab$ on a word of length $2^t$, which takes time $\calO(2^t)$; retrieving a (precomputed) matrix~$\mu(k)_{\equiv m}$, which takes constant time; and performing a multiplication and an additional reduction modulo~$m$, which takes time~$\calO(\log m \ \log\log m )$.

During the first $\lfloor\frac n{2^\sigma}\rfloor$ recursive calls, we have $t = \sigma$. This is followed by at most $\sigma$ additional calls for which $t < \sigma$ (the number of these calls is, essentially, governed by the binary expansion of $n - \lfloor \frac n{2^\sigma}\rfloor\,2^\sigma$).

Each of these first $\lfloor\frac n{2^\sigma}\rfloor$ recursive calls takes time $\calO(\lfloor\frac n{2^\sigma}\rfloor\,(2^\sigma + \sigma)) = \calO(\lfloor\frac n{2^\sigma}\rfloor\,2^\sigma) = \calO(n)$. The remaining $\sigma$ recursive calls take time at most $\calO\left(\sigma\,(2^{\sigma-1} + \sigma)\right)$, which is $\calO(\sigma\,2^\sigma)$. Since $2^\sigma \leqslant 2\frac{n^{1/2}}{\log n}$, these $\sigma$ recursive calls also take time $\calO(n)$. This concludes the proof.
\end{proof}

We now turn to the proof that the probability that $\mat(w)_{\equiv m} = \Id$ is $\calO(\log^{-2}n)$, where $n = |w|$ and $m$ is the integer computed above.

\begin{proposition}\label{prop: rare event}
Let $n \geqslant 1$.
Let $\sigma$ and $m$ be as above (where $\sigma$ and $m$ are functions of $n$).
Suppose that $G$ is infinite.
The probability that a word $w \in \tilde A^*$ chosen uniformly at random among the words (resp. the reduced words) of length $n$ satisfies $\mat(w)_{\equiv m} = \Id$ is $\calO(\frac1{\log^2 n})$.
\end{proposition}

\begin{proof}
Recall that $m = 2^\tau$, \ie, $\tau = \log m$, where $\tau = \lceil 2^\sigma / \sigma \rceil$, and that, by~\Cref{prop:tits}, either $G$ has exponential growth or it has polynomial growth of order $\delta$ for some $\delta \geqslant 1$.
We will then conclude using the upper bounds in \Cref{cor:rare-events}.

If $G$ has polynomial growth of order $\delta \geqslant 1$, \Cref{pro:time-of-tn} states that $2^\sigma = \Omega(n^{1/4\delta})$.
It follows that $\tau = \Theta(2^\sigma / \sigma) \subseteq \Omega(2^{\sigma/2})$ and that
\[\tau^{-1/4} \log^{3/4} \tau = \calO(\tau^{-1/5}) \subseteq \calO(2^{-\sigma/10}) \subseteq \calO(n^{-1/40\delta}) \subseteq \calO(\log^{-2} n).\]
It now follows from~\Cref{cor:rare-events} that
$\bbP[\mat(w)_{\equiv m} = \Id] = \calO(\tau^{-1/4} \log^{3/4} \tau)
\subseteq \calO(\log^{-2} n)$.

If, instead, $G$ has exponential growth, \Cref{pro:time-of-tn} states that $2^\sigma = \Theta(\log n)$.
It follows that $\tau = \Theta(2^\sigma / \sigma) \subseteq \Omega(2^{\sigma/2}) \subseteq \Omega(\log^{1/2} n)$: there exists a constant $c > 0$ such that $\tau > c \log^{1/2} n$.
Hence, it follows from~\Cref{cor:rare-events} that, for some constant $c' > 0$, we now have
\[\bbP[\mat(w)_{\equiv m} = \Id] \leqslant c' \exp\left(-\tau^{1/3} / c'\right) \leqslant c' \exp\left(-c^{1/3} \log^{1/6} n / c'\right) = \calO(\log^{-2} n). \qedhere\]
\end{proof}

We now summarize our \emph{Quick Word Problem} algorithm.

\medskip

\begin{algorithm}[H]
\caption{\bf Algorithm~$\QWP$}
\DontPrintSemicolon
\SetKwInOut{Input}{Input}
\SetKwInOut{Output}{Output}

\Input{a word~$w = a_1 a_2 \cdots a_n$ on alphabet~$\tilde A$}
\Output{\textsf{Yes} or \textsf{No}, depending on whether $\mat(w) = \Id$}
Run Algorithm $\ModC$ to compute $\mat(w)_{\equiv m}$\;
\lIf{$\mat(w)_{\equiv m} \ne \Id$}{\Return~\textsf{No}}
\Else{
Run Algorithm $\DC_A$ on $w$\;
\lIf{$\mat(w) = \Id$}{\Return~\textsf{Yes}}
\lElse{\Return~\textsf{No}}}
\end{algorithm}

\begin{theorem}\label{thm:WP linear average}
Algorithm~$\QWP$ solves $\WP_A$ in linear average time.
\end{theorem}

\begin{proof}
It is immediate that $\QWP$ solves $\WP_A$. The average time complexity of $\QWP$ satisfies 
$$\sfC^{(\QWP)}_\textsf{ac}(n) \leqslant \sfC^{(\ModC)}_\textsf{ac}(n) + \bbP_m\,\sfC^{(\DC_A)}_\textsf{wc}(n),$$
where $\bbP_{m,n}$ is the probability that $\mat(w)_{\equiv m} = \Id$ when $w$ is a random word of length $n$, and $\bbP_m = \sup_{n\geqslant m}\bbP_{m,n}$, see \Cref{cor:rare-events}.

If $G$ is finite, $\sfC^{(\QWP)}_\textsf{ac}(n)$ is $\calO(n)$ since both $\sfC^{(\ModC)}_\textsf{ac}(n)$ and $\sfC^{(\DC_A)}_\textsf{wc}(n)$ are (by \Cref{cor: complexity DC finite}).
If $G$ is infinite, we conclude using \Cref{cor: complexity DC} and \Cref{prop: rare event}.
\end{proof}

%%%%%%%%%%%%%%%%
\subsection{Special cases: finite, free, and nilpotent groups}

Here we briefly review the behavior of our algorithms when the group $G$ is finite or free. We also discuss the case where the matrices in $\mat(A)$ are upper-triangular (and hence unitriangular since they are assumed to be invertible in $\GL_d(\bbZ)$). The latter case yields a new bound on the complexity of the word problem in finitely generated nilpotent groups.

\paragraph{Finite groups}
In this case, \Cref{cor: complexity DC finite} established that $\DC_A$ has linear worst-case complexity, and $\QWP$ also does.

\paragraph{Free groups}
If $G$ is freely generated by $\mat(A)$, the answer to the word problem is always \textsf{No} if the input word $w$ is a non-empty reduced word. We do not know how to leverage this observation to give a better algorithm, since it is not known whether one can decide whether $\mat(A)$ freely generates $G$. In that case, Algorithm $\QWP$ will compute the sets $P_s$ ($s\leqslant \sigma$), each of which has cardinality $2^s$. Thus $\hat P_s = 2^{s+1}-1$, and $\sigma = 1 + \lfloor \log n - 2\rfloor$.

\paragraph{Groups generated by triangular matrices, nilpotent groups}
If $M$ and $M'$ are upper triangular matrices and if $1\leqslant i \leqslant j \leqslant d$, then the $(i,j)$-entry of $MM'$ is $\sum_{h=i}^jM_{i,h}\,M'_{h,j}$. As a result, if the matrices in $\mat(A)$ are upper-triangular (and hence so are their inverses), and if $w \in \tilde A^*$ has length $n$, then for all $1\leqslant i \leqslant j \leqslant d$, the $(i,j)$-entry of $\mat(w)$ is bounded above by a polynomial in $n$, with degree $j-i$. In particular, the length of the entries of $\mat(w)$ is logarithmic.

From \Cref{prop: basic facts}, for words $w$ and $w'$ of length at most $n$, the product $\mat(w)\mat(w')$ is computed in polylogarithmic time, so the Master Theorem (\Cref{prop: master theorem}) proves the linear worst-case complexity of $\DC_A$ and $\WP_A$.
As observed by Olshanskii and Shpilrain, combining this with the well-known fact~\cite{1938:Hirsch} that every finitely generated nilpotent group embeds in the direct product of a finite group and a group of upper-triangular matrices in some  $\GL_d(\sfZ)$ gives the following statement, improving on \cite[Theorems 1 and 2]{2025:OlshShp}.

\begin{proposition}\label{prop: upper triangular}
 If $\mat(A)$ consists only of upper (resp. lower) triangular matrices, then $\WP_A$ has linear worst-case complexity, using either Algorithm $\DC_A$ or Algorithm $\QWP$.

 If $G$ is a finitely generated nilpotent group, then the Word Problem in $G$ can be solved in linear worst-case complexity.
\end{proposition}

%%%%%%%%%%%%%%%%
\section{Matrices over larger rings}\label{sec: larger rings}

We consider below the case of matrices over rational numbers, and over algebraic number fields.

%%%%%%%%%%%%%%%%
\subsection{Matrices with coefficients in $\bbQ$}\label{sec: rational}

We identify a rational number $r$ with the corresponding reduced fraction $r = p/q$, with $p,q \in \bbZ$, $q > 0$ and $(p,q) = 1$ ($p = 0$ and $q = 1$ if $r = 0$). We define the \emph{length} of $r$ to be $\ell(r) = \ell(p) + \ell(q)$. If $M$ is a $d\times d$ matrix with rational coefficients, we define its \emph{length} $\ell(M)$ to be the maximum value of the lengths of its coefficients.

\paragraph{Bit complexity of arithmetic in~$\bbQ$}
Reducing a fraction to lowest terms requires computing a~$\gcd$, and this is more expensive than the elementary operations on integers discussed in \Cref{prop: basic facts}. Using the fast Euclidean algorithm,
the $\gcd$ of two integers of length at most~$L$ is computed in time $\calO(L\log^2L)$ \cite[Corollary~11.6, and the remark preceding Corollary~11.10]{vzGathenGerhard}. This is more expensive, by a factor~$\log L$, than the operations in \Cref{prop: basic facts}.

\begin{proposition}\label{prop: basic facts Q}
Let $p, q\in \bbZ$ with $q\ne 0$, and let $M,M'$ be $d\times d$ matrices over $\bbQ$.
\begin{enumerate}[(i)]
\item\label{item: simplify fraction} Reducing the fraction $p/q$ to lowest terms can be done in time $\calO(L\log^2L)$, where $L = \max(\ell(p), \ell(q))$.

\item\label{item: matrix product reduced} If $M,M'$ have reduced entries of length at most~$L$, then $\ell(MM') = \calO(L)$, and $MM'$, with reduced entries, is computed in time $\calO(L\log^2L)$.
\item\label{item: matrix product unreduced} If $M = N/q$ and $M' = N'/q'$, where $N,N'\in\bbZ^{d\times d}$ and $q,q'$ are positive integers of length at most~$L$ (not necessarily denominators in lowest terms), then $MM' = (NN')/(qq')$ has a numerator and denominator of length $\calO(L)$, computed, without reduction, in time~$\calO(L\log L)$.
\end{enumerate}
\end{proposition}

\begin{proof}
Item~\eqref{item: simplify fraction} is the fast $\gcd$ algorithm discussed above. Item~\eqref{item: matrix product unreduced} follows directly from \Cref{prop: basic facts}: computing $NN'$ and $qq'$ amounts to a bounded number ($\calO(d^3)$) of products and sums of integers of length $\calO(L)$, each costing $\calO(L\log L)$ by \Cref{prop: basic facts}~\eqref{item: HvdH}. Item~\eqref{item: matrix product reduced} follows by additionally reducing each of the $d^2$ resulting entries to lowest terms, at a cost $\calO(L\log^2L)$ each, by Item~\eqref{item: simplify fraction}.
\end{proof}

Thus, reducing a fraction $p/q$ to lowest terms, or computing the sum or product of two reduced fractions of length at most~$L$ and reducing the result, takes time $\calO(L\log^2L)$ rather than the $\calO(L\log L)$ complexity which prevails for integers. This extra logarithmic factor would degrade the worst-case complexity bound $\calO(n\log^2n)$ that we establish below for Algorithm~$\DC_A$ over~$\bbQ$ (\Cref{cor: complexity DC Q}) into $\calO(n\log^3n)$, if fractions were reduced to lowest terms after every intermediate matrix product. We avoid the extra factor in the proof of \Cref{cor: complexity DC Q} by working with a common denominator and postponing all reductions to a single, final step. More precisely, we adapt Algorithm~$\DC_A$ as follows.

\paragraph{Adapting Algorithm~$\DC_A$ to matrices over~$\bbQ$}
Reducing matrix entries to lowest terms after every product of Algorithm~$\DC_A$ would, by \Cref{prop: basic facts Q}~\eqref{item: matrix product reduced}, yield the functional equation $C(n) = C(\lfloor n/2\rfloor) + C(\lceil n/2\rceil) + \calO(n\log^2n)$ and hence, by \Cref{prop: master theorem}, the bound $C(n) = \calO(n\log^3n)$, presenting an extra logarithmic factor when compared with the algorithm over $\bbZ$. We avoid this extra factor by representing matrix values with a common denominator, and reducing (if needed) only once, at the very end.

Let $\Delta$ be a common denominator for the entries of the matrices $\mat(a)$, $a\in\tilde A$. For each such letter, let $\widehat\mat(a) = \Delta\,\mat(a)$, so that $\widehat\mat(a) \in \bbZ^{d\times d}$. We extend $a\mapsto \widehat\mat(a)$ to a monoid morphism $\widehat\mat\colon \tilde A^*\to \bbZ^{d\times d}$, so that $\mat(u) = \widehat\mat(u)/\Delta^{|u|}$ for every $u\in\tilde A^*$. Note that $\widehat\mat(u)$ need not be a reduced numerator (with respect to $\Delta^{|u|}$), and that $\widehat\mat(u)$ is in $\GL_d(\bbQ)$ but may not be in $\GL_d(\bbZ)$. Note also that, if $w$ has length $n$, then $\mat(w) = \Id$ if and only if $\widehat\mat(w) = \Delta^n\,\Id$.

We run Algorithm~$\DC_A$ representing every intermediate value $\mat(u)$ by the pair $(\widehat\mat(u), \Delta^{|u|})$, without ever reducing this fraction, using \Cref{prop: basic facts Q}~\eqref{item: matrix product unreduced} at each combination step (the two denominators $\Delta^{|u_1|}$ and $\Delta^{|u_2|}$ being powers of the same integer~$\Delta$).

This computes $\mat(w)$ in the unreduced form $\widehat\mat(w)/\Delta^n$. Deciding $\WP_A$, that is, whether $\mat(w) = \Id$, only requires checking whether $\widehat\mat(w) = \Delta^n\,\Id$. If a reduced representative of $\mat(w)$ is instead required (to solve $\EC_A$ with a reduced output), one reduces each of the $d^2$ entries of $\widehat\mat(u)/\Delta^n$ to lowest terms.

We record the complexity of this adaptation of Algorithm $\DC_A$ in \Cref{cor: complexity DC Q}. The worst-case bound was asserted without a detailed cost analysis in~\cite[Proposition 3]{2025:OlshShp}.

\begin{corollary}\label{cor: complexity DC Q}
Algorithm $\DC_A$ solves $\EC_A$ and $\WP_A$ for matrices in $\GL_d(\bbQ)$ with worst-case complexity $\calO(n \log^2n)$. If $G = \langle\mat(A)\rangle$ is finite, $\DC_A$ runs in time $\calO(n)$.
\end{corollary}

\begin{proof}
First we show that $\ell(\widehat\mat(u)) = \calO(|u|)$: writing $L(n) = \max\{\ell(\widehat\mat(u)) : |u| = n\}$, \Cref{prop: basic facts}~\eqref{fact:matrix-product} (which applies to arbitrary integer matrices, invertible or not) gives $L(n) \leqslant 2L(\lceil n/2\rceil) + \calO(1)$ for $n \geqslant 2$. An immediate induction yields $L(n) = \calO(n)$. Likewise $\ell(\Delta^n) = \calO(n)$, since $\Delta$ is a constant.

Consequently, if $w = w_1\cdots w_n$ and $m = \lfloor n/2\rfloor$, then $\widehat\mat(w_1\cdots w_m)$, $\widehat\mat(w_{m+1}\cdots w_n)$, $\Delta^m$ and $\Delta^{n-m}$ all have length $\calO(n)$, so that computing $(\widehat\mat(w), \Delta^n)$ from $(\widehat\mat(w_1\cdots w_m), \Delta^m)$ and $(\widehat\mat(w_{m+1}\cdots w_n), \Delta^{n-m})$ costs $\calO(n\log n)$ (one integer matrix product, by \Cref{prop: basic facts}~\eqref{fact:matrix-product}, and one integer product, by \Cref{prop: basic facts}~\eqref{item: HvdH}). The complexity~$C(n)$ of computing $(N_w,\Delta^n)$ this way therefore satisfies
\begin{equation*}
C(n) = C(\lfloor n/2\rfloor) + C(\lceil n/2\rceil) + \calO(n\log n)
\end{equation*}
for $n\geqslant 2$, and \Cref{prop: master theorem} yields $C(n) = \calO(n\log^2n)$. Reducing entries to lowest terms only once, at the very end, costs an additional $\calO(n\log^2n)$ by \Cref{prop: basic facts Q}~\eqref{item: simplify fraction} (to solve $\EC_A$), and comparing $\widehat\mat(w)$ to $\Delta^n\,\Id$ costs an additional $\calO(n)$ by \Cref{prop: basic facts}~\eqref{item: comparing integers} (to solve $\WP_A$). Neither affects the announced bound.
\end{proof}

To obtain a linear average-case complexity, we also adapt Algorithm~$\QWP$ as follows.

\paragraph{Adapting Algorithm~$\QWP$ to matrices over~$\bbQ$}
For each integer $s\geqslant 0$, let $P_s = \widehat\mat(\tilde A^{2^s})$ and $\hat P_s = \sum_{s'\leqslant s}|P_{s'}|$. Algorithm $\Bounds_n$ works as before, to compute $\sigma$, $\mu$ and $\sfL$ such that 
\begin{itemize}
\item {}$\sigma$ is maximal such that $2^\sigma \leqslant n^{1/2}/\log n$ and $\hat P_\sigma \leqslant n^{1/4}$;

\item {}$\mu$ establishes a bijection from $[\hat P_{s-1}+1, \hat P_s]$ to $P_s$ for every $s\leqslant \sigma$;

\item and $\sfL$ maps $(k_1,k_2) \in [\hat P_{s-1}+1, \hat P_s]^2$ to the element $k\in [\hat P_s+1, \hat P_{s+1}]$ such that $\mu(k) = \mu(k_1)\,\mu(k_2)$ (for each $s < \sigma$).
\end{itemize}
The upper and lower bounds estimates of $\sigma$ in \Cref{pro:time-of-tn} apply, as well as the fact that the worst-case complexity of $\Bounds_n$ is $\calO(n)$.

The choice of the modulus $m$ must be slightly modified: we let $\tau = \lceil\frac{2^\sigma}\sigma\rceil$ and we choose $m \in [2^\tau, 2^\tau+\Delta-1]$ to be relatively prime with $\Delta$, so that $\Delta$ is invertible mod $m$. Then $m = \Theta(\frac{2^\sigma}\sigma)$ and, in particular, $m = \calO(n^{1/2})$. As a result, and in view of \Cref{prop: basic facts}, the analogue of \Cref{lem: compute some mat(w) mod m} holds as well: one can compute the $\mu(k)_{\equiv m}$ ($k \leqslant \hat P_\sigma$) in time $\calO(n)$.

The next step of the algorithm, namely Algorithm $\ModC$ to compute $\widehat\mat(w)_{\equiv m}$, for every $w\in \tilde A^*$ of length $n$, works as above, and in linear time (\Cref{prop: complexity of calM}). Problem $\WP_A$ is then solved, as expected, by first computing $\widehat\mat(w)_{\equiv m}$ and verifying whether it is equal to $\Delta_{\equiv m}^n\,\Id$. If not, then $\mat(w) \ne \Id$. If instead $\widehat\mat(w)_{\equiv m} = \Delta_{\equiv m}^n\,\Id$, one runs Algorithm $\DC_A$ to compute $\mat(w)$ and one compares that matrix to $\Id$.

To establish that this represents an algorithm with linear average-case complexity, we need to verify the analogue of \Cref{prop: rare event}, namely to prove that the probability that
\begin{equation}\label{eqref: proba mod m}
\bbP[\widehat\mat(w)_{\equiv m} = (\Delta^n)_{\equiv m}\,\Id] = \calO(\log^{-2}n).
\end{equation}
This uses the following analogue of \Cref{cor:rare-events}, where $\mat$ has been replaced by $\widehat\mat$.
\begin{corollary}\label{lemma: rare events for rational}
If $G = \langle \mat(A)\rangle$ is infinite, there exists a constant $c>0$ such that the probability that a random word (resp. random reduced word) $w$ of length $n$ satisfies $\widehat\mat(w)_{\equiv m} = (\Delta^n)_{\equiv m}\,\Id$ is bounded above by $c\, \log^{-1/4}m\, \log^{3/4}\log m$. This bound is $c\, \exp\left(-\frac{\log^{1/3}m}c\right)$ if $G$ has exponential growth.
\end{corollary}
Indeed, it suffices to adapt the proof of~\Cref{cor:rare-events}, by simply observing that $\widehat\mat(w)_{\equiv m} = (\Delta^n)_{\equiv m} \Id$ if and only if $\mat(x) = (\Delta^n \mat(y)^{-1})_{\equiv m}$, where $xy$ is the factorisation of $w$ into a prefix $x$ of length $p$ and a suffix $y$ of length $q$;
the rest of the proof applies verbatim.

The proof of Equation~\eqref{eqref: proba mod m} then proceeds as in the case of matrices over $\bbZ$, using \Cref{lemma: rare events for rational} instead of \Cref{cor:rare-events}.

By following the same proof as before, we can now establish the following complement to \Cref{thm:WP linear average}.

\begin{theorem}\label{prop: WP linear over Q}
The Word Problem can be solved in linear average time for finitely generated subgroups of $\GL_d(\bbQ)$.
\end{theorem}

Since finitely generated polycyclic groups can be explicitly represented as subgroups of $\GL_d(\bbQ)$ \cite{1967:Auslander}, \Cref{cor: complexity DC Q} and \Cref{prop: WP linear over Q} yield the following corollary.

\begin{corollary}\label{cor: polycyclic}
The Word Problem in a finitely generated polycyclic group can be solved with worst-case complexity $\calO(n\log^2n)$ and with linear average-case complexity.
\end{corollary}

When, in addition, the matrices in $\mat(A)$ are unitriangular ($G$ is then nilpotent) the argument of \Cref{prop: upper triangular} extends to $\bbQ$ and yields a stronger, linear worst-case bound.

\begin{proposition}\label{prop: upper triangular Q}
If $\mat(A)$ consists only of upper (resp. lower) unitriangular matrices over $\bbQ$, then $\WP_A$ has linear worst-case complexity, using either Algorithm $\DC_A$ or Algorithm $\QWP$.
\end{proposition}

\begin{proof}
If $\mat(A)$ consists of upper-unitriangular matrices, let $\Delta$ be a common denominator for their entries, so that $\widehat\mat(a) = \Delta\,\mat(a) \in\bbZ^{d\times d}$ is upper-triangular for every $a\in\tilde A$, with $\Delta$ on the diagonal (since $\mat(a)$ has diagonal entries equal to~$1$). If $w\in\tilde A^*$ has length $n$, the same computation as in \Cref{prop: upper triangular} shows that, for $1\leqslant i\leqslant j\leqslant d$, the $(i,j)$-entry of $\Delta^{j-i}\mat(w)$ is an integer bounded above by a polynomial in $n$ of degree $j-i$. Hence $\ell(\mat(w)) = \calO(\log n)$, and \Cref{prop: basic facts Q} together with the Master Theorem (\Cref{prop: master theorem}) again yield a linear worst-case complexity for $\DC_A$ and $\WP_A$.
\end{proof}

%%%%%%%%%%%%%%%%
\subsection{Matrices with coefficients in larger rings}\label{sec: ring of integers}

We now consider matrices with entries in rings $R$ other than $\bbZ$ or $\bbQ$, but still closely tied to them, for instance the ring of integers of an algebraic number field, or the algebraic field itself. 
The Word Problem is formulated in $\GL_d(R)$ as in $\GL_d(\bbZ)$: a fixed finite non-empty set $A$ is given, together with a map $\mat\colon A \to \GL_d(R)$ and the values of the matrices $\mat(a)\inv$ ($a\in A$); the symmetrized alphabet $\tilde A$ is defined as before, and $\mat$ is extended to a map from $\tilde A^*$ to $\GL_d(R)$. Given a word $w\in \tilde A^*$, one must decide whether $\mat(w) = \Id$.

As it turns out, such extensions reduce directly to the cases $R=\bbZ$ or $R=\bbQ$ already treated, via an explicit embedding, which we now make precise.

Let $R$ be a ring and let $S$ be a subring contained in the center of $R$. Then $R$ can be seen as an $S$-module. Suppose also that $R$ is free and of finite rank as an $S$-module, that we are given a basis $(\alpha_1,\dots,\alpha_\delta)$ of $R$, and that we are also given the expression of the products $\alpha_i\alpha_j$ ($1\leqslant i,j\leqslant\delta$) in that basis, say, $\alpha_i\,\alpha_j = \sum_k \beta_{i,j,k}\alpha_k$ with the $\beta_{i,j,k} \in S$.

For any element $r\in R$, let $\rho(r)$ be the left multiplication by $r$. Under our hypotheses, $\rho(r)$ is a linear map. Adopting the matrix notation for linear maps (with respect to the basis $(\alpha_1,\dots,\alpha_\delta)$), we have an injective and unital ring morphism $\rho\colon R \to M_\delta(S)$, where $\rho(\alpha_i) = (\beta_{i,j,k})_{1\le j, k\le \delta}$. This map naturally extends to an injective and unital ring morphism $\hat\rho\colon M_d(R) \to M_{d\delta}(S)$, where each entry of $M_{i,j}$ of $M \in M_d(R)$ is replaced by $\rho(M_{i,j})$, a $\delta\times\delta$ block over $S$. It is immediate that, if $M$ is invertible in $M_d(R)$, then so is $\hat\rho(M)$ (since $M\,M' = \Id_d$ implies $\hat\rho(M)\,\hat\rho(M') = \Id_{d\delta}$).

\begin{proposition}
 With $R$ and $S$ as above, and given a map $\mat\colon A \to GL_d(R)$, the word problem $\WP_A$ over $\GL_d(R)$ reduces in constant time to the problem $\WP_A$ over $\GL_{d\delta}(S)$, relative to the map $\hat\rho\circ\mat$.
\end{proposition} 

\begin{proof}
 With $\rho$ and $\hat\rho$ defined as above, let $\mat'\colon A \to \GL_{d\delta}(S)$ be given by $\mat'(a) = \hat\rho(\mat(a))$ for each $a\in A$. We now note that, for any word $w$, $\mat(w) = \Id_d$ if and only if $\hat\rho(\mat(w)) = \Id_{d\delta}$ (since $\hat\rho$ is an injective and unital ring morphism), if and only $\mat'(w) = \Id_{d\delta}$. This states exactly that the word problem relative to $\mat$ reduces to the word problem relative to $\mat'$. Since computing the map $\mat'$ on $A$ is performed in constant time, this is a constant time reduction.
\end{proof}

Our earlier results on the complexity of the Word Problem over matrices in $\bbZ$ and in $\bbQ$ then yield the following corollary.

\begin{corollary}
 Let $S \in \{\bbZ, \bbQ\}$ and let $R$ be a ring containing $S$ in its center, which is free and has finite rank as an $S$-module. Suppose that we are given a basis $(\alpha_1,\dots,\alpha_\delta)$ of $R$ over $S$ and the expression of the $\alpha_i\,\alpha_j$ ($1\le i,j \le\delta$) in that basis. Then the Word Problem $\WP_A$ over $\GL_d(R)$ can be solved in $\calO(n\log^2n)$ worst-case complexity and in linear average-case.
\end{corollary}

 This applies in particular to the cases where $R$ is an algebraic number field, the ring of integers of such a field, a direct product of finitely many copies of $S$.

\begin{remark}
    Let $R_1$ and $R_2$ be rings. It is immediate that the word problem in groups of matrices over $R_1\times R_2$ reduces in constant time to the word problems in groups of matrices over $R_1$ and over $R_2$ (via the projection morphisms of $R_1\times R_2$ onto $R_1$ and $R_2$, extended to matrices). Thus, if the worst-case (resp. average-case) complexity of $\WP_A$ is $\calO(f(n))$ in groups of matrices over $R_1$ and over $R_2$, then the same holds for matrices over $R_1\times R_2$.

    Note that, if $R$ is a finite ring, then $\GL_d(R)$ is finite as well, and $\WP_A$ over $\GL_d(R)$ is has linear worst-case complexity. As a result of these two observations, $\WP_A$ has worst-case complexity $\calO(n\log^2n)$ and linear average case complexity on groups of matrices over $\bbZ^k \times \bbQ^\ell$, $\bbZ \times \bbZ/p\bbZ$, $\bbQ\times \bbZ/p\bbZ$, etc.
\end{remark}

\section{Lower bounds}\label{sec: lower bound}

Let $A$, $\tilde A$ and $\mat\colon A\to \GL_d(\bbZ)$ be as in \Cref{section:algorithms}.

In Section~\ref{section:algorithms}, we proved that, in general, one can check in time linear in~$n$ whether a word chosen uniformly at random among words (resp. reduced words) of length~$n$ over the finite set~$\mat(\tilde A)$ of invertible~$d \times d$ matrices yields the identity matrix;
this improved the previously known worst-case upper bound of~$\calO(n \log^2n)$.
However, \emph{a priori}, the word problem might be solved by other algorithms with an even better complexity.
This is the case, for instance, when~$\mat(A) = \{-\Id\}$, in which case~$\mat(w) = \Id$ if and only if~$|w|$ is even.
However, we argue below that under very mild hypotheses, no algorithm can decide the word problem in sublinear time for all values of~$n$, even in the \emph{best} case. Formally, we show that, under mild hypotheses on $A$ and $\mat$, if $w$ is any word of even length $2n$ and $X$ is any set of at most $n$ positions in $\{1,\dots, 2n\}$, there exist words $x$ and $y$ of length $2n$, both coinciding with $w$ on positions in $X$, and such that $\mat(x) = \Id$ while $\mat(y) \ne \Id$. It follows that any algorithm (deterministic, probabilistic, etc) reading only half of the input will fail to solve $\WP_A$.

\begin{proposition}\label{pro:lower-bound-1}
If~$\mat$ is non-constant on $\tilde A$, checking whether an arbitrary input word~$w \in \tilde A^*$ of length $2n$ satisfies $\mat(w) = \Id$, requires reading more that $n$ letters of~$w$.
\end{proposition}

\begin{proof}
Let~$X$ be an arbitrary subset of~$\{1,2,\ldots,2n\}$ of cardinality~$|X| = n$.
We first prove by induction on~$n$ that there exists a word~$x \in \tilde A^{2n}$ such that~$\mat(x) = \Id$ and such that~$w_i = x_i$ for all~$i \in S$.

Indeed, if~$n = 0$, there is nothing to do.
If~$n \geqslant 1$, there exist consecutive integers, say~$k$ and~$k+1$, exactly one of which belongs to~$X$.
We set $a = w_k$ if~$k \in X$ and $a = \bar w_{k+1}$ if~$k+1 \in X$.

Let~$E = \{1,2,\ldots,2n\} \setminus \{k,k+1\}$, and let~$\phi \colon E \to \{1,2,\ldots,2n-2\}$ be the bijection given by $\phi(i) = i$ if $i<k$ and $\phi(i) = i-2$ if $i > k+1$.
Let also $X' = \phi(X \cap E)$ and $w' = w_1 w_2 \cdots w_{k-1} w_{k+2} w_{k+3} \cdots w_{2n}$: we have $|X'| = n-1$ and $|w'| = 2(n-1)$. By induction, there exists a word~$x'$ such that~$w'_i = x'_i$ for all~$i \in X'$, and~$\mat(x') = \Id$.
Now the word
\[x = x'_1 x'_2 \cdots x'_{k-1} a\, \bar a\, x'_k x'_{k+1} \cdots x'_{2n}\]
coincides with~$w$ at each index in $X$, and satisfies the equalities~$\mat(x) = \mat(x') = \Id$, which completes the induction.

There also exists another word~$y \in \Sigma^{2n}$ such that~$\mat(y) \neq \Id$ and such that~$y_i = w_i$ for every~$i \in X$: it suffices to choose $y_i = x_i$ for every position $i$, except for a position $j\notin X$, where we choose $y_j$ such that $\mat(y_j) \neq \mat(x_j)$.

In conclusion, reading only~$n$ letters of a word~$w$ does not suffice to decide whether~$\mat(w) = \Id$.
\end{proof}

The following is an analogous statement for reduced words.

\begin{proposition}\label{pro:lower-bound-2}
Assume that~$|A| \geqslant 2$ and $\mat(A)$ does not freely generate $G$, and let $z$ be a non-empty cyclically reduced word~$z$ such that~$\mat(z) = \Id$.
Checking whether an arbitrary reduced word~$w$ of length~$2n$ over~$\tilde A$ satisfies $\mat(w) = \Id$ requires reading at least~$\lfloor 2n / (8|z|+17) \rfloor$ letters of~$w$.
\end{proposition}

\begin{proof}
Let $k = |z|$, $z = z_1\cdots z_k$ (with each $z_i\in\tilde A$) and let $\kappa = 2k+4$. 
Say that a word $x$ is an \emph{extension of a letter $a\in \tilde A$} if $x$ is reduced, $|x| = \kappa$, $\mat(x) = \Id$, $x$ starts with $a$ and ends with $\bar a$.
We first observe that every letter $a$ admits an extension.
Indeed, if there exists an index $1\leqslant i < k$ such that $z_i \neq a$ and~$z_{i+1} \neq \bar{a}$, then the word~$a^2 z_{i+1} \cdots z_k z z_1 \cdots z_i \bar{a}^2$ is an extension of $a$.
Similarly, if $z_k \neq a$ and~$z_1 \neq \bar{a}$, then $a^2 z^2 \bar{a}^2$ is an extension of $a$.
If neither of these situations hold, then $z$ is a power of $a$ or $\bar a$, and if $b \in \tilde A \setminus \{a,\bar{a}\}$, then $a b z^2 \bar b \bar{a}$ is an extension of $a$.

Since \Cref{pro:lower-bound-2} is trivial when $2n < 8|z|+17 = 4\kappa+1$, we assume that $n \geqslant 4\kappa+1$.
Proceeding as for \Cref{pro:lower-bound-1}, we prove by induction on $n$ that, for each reduced word $w = w_1\cdots w_{2n} \in \tilde A^{2n}$ and for each subset $X \subseteq \{1,\dots, 2n\}$ of cardinality $|X| < \lfloor 2n/(8k+17) \rfloor = \lfloor 2n/(4\kappa+1) \rfloor$, there exist reduced words $x$ and $y$ of length $2n$, which coincide with $w$ on each position in $X$ and such that $\mat(x) = \Id$ and $\mat(y) \ne \Id$. Let us start with the existence of $x$.

Such a word exists if $X = \emptyset$ and for any $n\geqslant \kappa$: fix $a\in \tilde A$ and let $u$ be an extension of $a$;
then $x = a^{n-\kappa/2}\,u\,\bar a^{n-\kappa/2}$ satisfies the required property.

We now suppose that $X$ is not empty.
The union of the intervals $[i-2\kappa,i+2\kappa]$, when $i$ runs over $X$, contains at most $(4\kappa+1)|X|$ elements.
Since $(4\kappa+1) |X| < |w|$, there exists an element $i_0\in [1,2n]$ at distance greater than $2\kappa$ from $X$.
We distinguish the following cases: (i) $i_0 > \min X$; (ii) $i_0 < \max X$.
We note that if neither (i) nor (ii) holds, we have $\max X \leqslant i_0 \leqslant \min X$, \ie, $X = \{i_0\}$, which is absurd.

In Case (i), let $j_0$ be the predecessor of $i_0$ in $X$, that is, $j_0 = \max(X \cap [1,i_0-1])$. We note that $j_0+2\kappa < i_0$. Let $a$ be a letter distinct from $\bar w_{j_0-1}$ (if $j_0\geqslant 2$) and $\bar w_{j_0+2\kappa+1}$.  Let then
$$w' = w_1\cdots w_{j_0-1}\,a\,w_{j_0 + 2\kappa + 1} \cdots w_{2n}\textrm{ and }X' = ([1,j_0-1] \cap X) \cup \{j - 2\kappa : j\in X \cap [j_0+2\kappa+1,2n]\}.$$
By construction, $|X'| = |X|-1$, $w'$ is reduced, and $|w'| = 2n -4\kappa$ (an even number). Moreover,
$$\Bigl\lfloor \, \frac{|w'|}{4\kappa+1} \, \Bigr\rfloor= \Bigl\lfloor \,\frac{2n-2\kappa}{4\kappa+1}\, \Bigr\rfloor \geqslant \Bigl\lfloor \, \frac{|w|}{4\kappa+1} - 1 \, \Bigr\rfloor > |X|-1 = |X'|.$$
By induction, there exists a word $x'$ of length $|w'| = 2n-2\kappa$, which coincides with $w'$ on positions in $X'$, and such that $\mat(x') = \Id$. Let $b\in \tilde A$ be distinct from $w_{j_0}$ and $x'_{j_0}$, let $u$ be an extension of $w_{j_0}$, let $v$ be an extension of $b$, and let
$$x = x'_1\dots x'_{j_0-1}\, u\, v\, x'_{j_0} \dots x'_{2n-2\kappa}.$$
Then $x$ is reduced, $\mat(x) = \Id$, $|x| = |x'| + 2\kappa = |w'| + 2\kappa = 2n = |w|$, and the letter of $x$ in position $j_0$ is $w_{j_0}$, so $x$ coincides with $w$ on the positions in $X$.

Case (ii) is handled symmetrically, and this concludes the proof of the existence of $x$.

We now compute a word $y$ of length $2n$, which coincides with $w$ (and $x$) on positions in $X$ and satisfies $\mat(y) \ne \Id$. Since $|X| < \lfloor 2n/(8|z|+17) \rfloor \leqslant \lfloor 2n/25 \rfloor$, and therefore $n \geqslant 13$ too, we have in particular $3|X| + 2 < 2n$.
It follows that there exist three successive positions $i-1$, $i$ and $i+1$ outside $X$.
Since $\mat$ is not constant on $\tilde A$, there exists $b\in \tilde A$ such that $\mat(b) \ne \mat(x_i)$. Using the fact that $|A| \geqslant 2$, we can select letters $a, c\in \tilde A$ such that $a\notin \{\bar b, \bar x_i, \bar x_{i-2}\}$ (just $a\notin \{\bar b, \bar x_i\}$ if $i = 2$) and $c\notin \{\bar b, \bar x_i, \bar x_{i+2}\}$ (just $c\notin \{\bar b, \bar x_i\}$ if $i = 2n-1$).

If $\mat(a) \ne \mat(x_{i-1})$, the word $y = (x_1\dots x_{i-2})\,a\,(x_i\dots x_{2n})$ is reduced, it differs from $x$ in a single position which is not in $X$, and $\mat(y) \ne \mat(x) = \Id$. Symmetrically, if $\mat(c) \ne \mat(x_{i+1})$, the word $y = (x_1\dots x_i)\,c\,(x_{i+2}\dots x_{2n})$ satisfies the required property. In the remaining case, where $\mat(a) = \mat(x_{i-1})$ and $\mat(c) = \mat(x_{i+1})$,
we let $y = (x_1\dots x_{i-2})\,(abc)\,(x_{i+2}\dots x_{2n})$ to conclude the proof.%
\end{proof}

\begin{remark}
Propositions~\ref{pro:lower-bound-1} and~\ref{pro:lower-bound-2} hold, with the same proofs, for a morphism $\mat\colon A\to G$ into an \emph{arbitrary} group $G$; in particular for $\mat\colon A \to \GL_d(R)$ with $R$ an arbitrary ring. Indeed, neither proof uses any arithmetic property of the matrix entries: both rely only on group-theoretic hypotheses on $\mat$ and $A$, namely that $\mat$ is non-constant on $\tilde A$ (for \Cref{pro:lower-bound-1}), and that $|A|\geqslant 2$ and $\mat(A)$ does not freely generate $G$ (for \Cref{pro:lower-bound-2}). 
\end{remark}

\bibliographystyle{alphaurl}
\newcommand{\etalchar}[1]{$^{#1}$}

\end{document}